\documentclass[11pt]{article} 

\usepackage{fullpage}  
\usepackage{graphicx}
\usepackage{mathtools}

\usepackage[driverfallback=dvipdfm]{hyperref}
\usepackage{pb-diagram}  		

\usepackage{url}

\usepackage{algorithm} 
\usepackage{algorithmic} 
\usepackage{multirow} 
\usepackage{hhline}  
\usepackage{amsmath}
\usepackage{xcolor}

\usepackage{graphicx}
\usepackage{amscd}
\usepackage{amssymb,mathrsfs}
\input{epsf.sty}
\usepackage{amsthm,amscd}
\usepackage{color}
\usepackage{latexsym}
\usepackage{epic}
\usepackage{appendix}
\usepackage{enumerate}
\usepackage{longtable}
\usepackage{lscape}
\usepackage{extarrows}
\usepackage{enumitem}
\usepackage{verbatim}
\usepackage{epstopdf}

\newlength\myindent

\newcommand{\comm}[1]{}

\newcommand{\pacomm}[1]{{}}

\newcommand{\delete}[1]{}

\newcommand\defeq{\coloneqq}

\newcommand{\Rmnum}[1]{\expandafter\@slowromancap\romannumeral #1@}
\newcommand{\inner}[3][]{{\left\langle #2,#3 \right\rangle_{#1}}}

\newtheorem{definition}{Definition}[section]
\newtheorem{theorem}{Theorem}[section]
\newtheorem{lemma}{Lemma}[section]
\newtheorem{corollary}{Corollary}[section]

\numberwithin{equation}{section}

\DeclareMathOperator{\diag}{\mathrm{diag}}

\DeclareMathOperator{\sspan}{\mathrm{span}}

\DeclareMathOperator{\relu}{\mathrm{relu}}

\newcommand{\R}{\mathbb{R}}

\newcommand{\eps}{\epsilon}

\newcommand{\xo}{x_*}

\begin{document}

\title{A Provably Convergent Scheme  for Compressive Sensing under Random Generative Priors
}


\author{Wen Huang \footnotemark[2] \footnotemark[6] \and
        Paul Hand \footnotemark[3]  \and
        Reinhard Heckel \footnotemark[4]\and
        Vladislav Voroninski  \footnotemark[5]
}

\renewcommand{\thefootnote}{\fnsymbol{footnote}}

\footnotetext[2]{Department of Mathematical Sciences, Xiamen University, China}
\footnotetext[3]{Department of Mathematics and College of Computer and Information Science, Northeastern University, USA.}
\footnotetext[4]{Department of Electrical and Computer Engineering, Rice University, USA.}
\footnotetext[5]{Helm.ai, USA}
\footnotetext[5]{Corresponding author. E-mail: huwst08@gmail.com.}


\newcommand\genericf{f}

\date{December 10, 2018}

\maketitle

\begin{abstract}
Deep generative modeling has led to new and state of the art approaches for enforcing structural priors in a variety of inverse problems.  In contrast to priors given by sparsity, deep models can provide direct low-dimensional parameterizations of the manifold of images or signals belonging to a particular natural  class, allowing for recovery algorithms to be posed in a low-dimensional space.  This dimensionality may even be lower than the sparsity level of the same signals when viewed in a fixed basis.  What is not known about these methods is whether there are computationally efficient algorithms whose sample complexity is optimal in the dimensionality of the representation given by the generative model.  In this paper, we present such an algorithm and analysis.  Under the assumption that the generative model is a neural network that is sufficiently expansive at each layer and has Gaussian weights, we provide a gradient descent scheme and prove that for noisy compressive measurements of a signal in the range of the model, the algorithm converges to that signal, up to the noise level.  The scaling of the sample complexity with respect to the input dimensionality of the generative prior is linear, and thus can not be improved except for constants and factors of other variables.  To the best of the authors' knowledge, this is the first recovery guarantee for compressive sensing under generative priors by a computationally efficient algorithm.

\end{abstract}

\section{Introduction}

Generative models have greatly improved the state of the art in computer vision and image processing, including inpainting, super resolution, compression, compressed sensing, image manipulation, MRI imaging, and denoising \cite{GAN,oord2016pixel,mao2016image,yeh2017semantic,sonderby2016amortised,ledig2016photo,johnson2016perceptual,mousavi2015deep,mousavi2017learning,zhu2016generative,rippel2017real,mardani2017deep,mardani2017recurrent,mousavi2015deep,mousavi2017learning}.
At the heart of this success is the ability of generative models to sample from a good approximation of the manifold of natural signals or images relevant for a particular task.  These models can be efficiently learned, in an unsupervised way, from a collection of examples of images of that signal class.  The progress in generative modeling over just the past two years has been immense; as an example, multiple methods can now generate synthetic photorealistic images of celebrity faces \cite{karras2017progressive,kingma2018glow}.  Generative models have also been trained in other contexts, and the quality of such models is expected to keep improving.

A reason for the success of some generative models in inverse problems is that they can efficiently map a low-dimensional space, called a latent code space, to an estimate of the natural signal manifold in high-dimensional signal space.  That is, generative models can provide an explicit parameterization of an approximation of the natural signal manifold.  This property then allows signal recovery problems to be posed in a low dimensional space.  In contrast, a standard structural model for natural signals over the past decade or so has been sparsity in an appropriate basis, such as a wavelet basis.  This perspective leads to a high dimensional optimization problem that cannot directly be solved and instead relies on a convex relaxation.  While this relaxation has been quite successful in linear regimes, such as compressive sensing, it has not yet been effective in nonlinear regimes, such as compressive phase retrieval.

In comparison to sparsity-based methods, generative methods can permit lower dimensional representations of some signal classes.  To see how lower dimensional representations are possible, consider the following toy model.  Consider a 1-parameter family  of high resolution natural images of a toy train rolling down wooden tracks.  The representation of each image in terms of a wavelet bases will be approximately sparse, relative to the image dimensionality, but accurate reconstruction will still require many wavelet coefficients.  In contrast, if that collection of images was directly modeled as a one-dimensional manifold, recovery should be possible with approximately 1 measurement.

In comparison to sparsity based methods, generative methods can be exploited more efficiently in some contexts.  For example, in compressive phase retrieval \cite{hand2018phase} show that the empirical risk objective under suitable assumptions has a favorable optimization landscape, in that there are no spurious local minima, when the number of generic measurements is linear in the latent dimensionality of a random generative model.  In contrast, no algorithm for compressive phase retrieval is known to succeed under less than $O(s^2)$ generic measurements, where $s$ is the signal sparsity.

In the context of compressive sensing from linear measurements and under generative priors, recovery can be posed as a nonconvex empirical risk optimization, which can be solved by first order gradient methods.  When solved this way, generative models have been shown to empirically outperform sparsity models in the sense that they can give comparable reconstruction error with 5-10x fewer compressive measurements in some contexts.  This empirical result indicates both that representations from generative models are low dimensional and can be efficiently exploited.  Nonetheless, this observation does not have a firm theoretical footing.  In principle, such gradient algorithms for nonconvex programs could get stuck in local minima.  Thus, it is important to provide algorithms that provably recover the underlying signal.



In this paper, we introduce a gradient descent algorithm for empirical risk minimization under a generative network, given noisy compressive measurements of its output.  We prove that if the network is random, the size of each layer grows appropriately, there are a sufficient number of compressive measurements, and the magnitude of the noise is sufficiently small, then the gradient descent algorithm converges to a neighborhood of the global optimizer and the size of the neighborhood only depends on the magnitude of the noise. In particular, the gradient descent algorithm converges to the global minimizer for noiseless measurements.  To the best of our knowledge, this is the first recovery guarantee for compressive sensing under a generative neural network model. Using numerical experiments, we empirically verify recovery up to the noise level, and in particular exact recovery in the noiseless case.

The justification for studying random networks is as follows.  First, the weights of some neural networks trained on real data exhibit statistics consistent with Gaussians.  Second, the theory for inverse problems under generative priors is nascent and challenging even for Gaussian networks.  Third, it is not immediately clear what model for the weights of trained generative models is most realistic while maintaining mathematical tractability.  And finally, random neural networks have recently been shown to be useful for image processing tasks \cite{ulyanov2017deep,heckel2018deep}; thus, any theoretical analysis of them may be relevant for those contexts.

\subsection{Relation to previous theoretical work}

A first theoretical analysis of compressive sensing under a generative prior appeared in \cite{BJPD2017}.  In that work, the authors studied the task of recovering a signal near the range of a generative network by the same nonconvex empirical risk objective as in the present paper.  They establish that if the number of measurements scales linearly in the latent dimensionality, then \textit{if one can solve to global optimality}  the nonconvex empirical risk objective, then one recovers the signal to within the noise level and representational error of the network.  Because the objective is nonconvex, and nonconvex problems are NP-hard in general, it is not clear that any particular computationally efficient optimization algorithm can actually find the global optimum.  That is, it is possible that any particular numerically efficient optimization algorithm gets stuck in local minima. In the present paper, we provide a specific computationally efficient numerical algorithm and establish a recovery guarantee for compressive sensing under generative models that satisfy suitable architectural assumptions.

A recent paper by a subset of the authors~\cite{HV17} provides a global analysis of the nonconvex empirical risk objective below for expansive Gaussian networks.
The paper shows that, under appropriate conditions, there are descent directions, of the nonconvex objective, outside neighborhoods of the global optimizer and a negative multiple thereof in the latent code space. That work, however, does not provide an analysis of the behavior of the empirical risk objective within these two neighborhoods, a specific algorithm, a proof of convergence of an algorithm, or a principled reason why the negative multiple of the global optimizer would not be returned by a naively applied gradient scheme.  Additionally, that work does not study noise tolerance.  Each of these aspects require considerable technical advances, for example establishing a nontrivial convexity-like property near the global minimizer.

The paper~\cite{ALM2015} presents a simple layer-wise inversion process for neural networks.  In the current setting, this result is not applicable because the final compressive layer can not be directly inverted without structural assumptions.  Instead, in the present paper, we analyze the inversion of the compressive measurements and the generative network together.

\section{Problem statement}

We consider a generator
$G\colon \R^k \to \R^n$ with $k \ll n$, given by a $d$-layer network of the form
$$
G(x) = \relu(W_d \ldots \relu(W_2 \relu( W_1 x)) \ldots ),
$$
where $\relu(a) = \max(a, 0)$ applies entrywise, $W_i \in \R^{n_i \times n_{i - 1}}$ are the weights of the network,  and $n_0 = k$ and $n_d=n$ are, respectively, the dimensionality of the input and output of $G$.  This model for $G$ is a $d$-layer neural network with no bias terms.
Let $\mathfrak{y}_* = G(x_*) \in \R^n$ be an image in the range of the generator $G$,  and let $A \in \R^{m \times n}$ be a measurement matrix, where typically $m \ll n$.

Our goal is to estimate the image $\mathfrak{y}_*$ from noisy compressive measurements $y = A G(x_*) + e$, where $A$ and $G$ are known and $e \in \mathbb{R}^m$ is an unknown noise vector.  To estimate this image, we first estimate its latent code, $x$, and then compute $G(x)$.
In order to estimate  $x_*$, we consider  minimization of  the empirical risk
\begin{equation}
\min_{x \in \R^k} \frac{1}{2}\|A G(x) - y\|^2. \label{GA:e1}
\end{equation}
For notational convenience, we let $W_{+, x}$ denote the matrix obtained by zeroing out the rows of $W$ that do not have a positive dot product with $x$, i.e.,
$$
W_{+, x} = \diag(W x > 0) W,
$$
where $\diag(w)$ denote the diagonal matrix whose $i,i$th entry is 1 if $w_i$ is true and 0 otherwise. with the entries of a vector $w$ on the diagonal.
Furthermore, we define $W_{1, +, x} = (W_1)_{+, x} = \diag(W_1 x > 0) W_1$ and
$$
W_{i, +, x} = \diag(W_i W_{i - 1, +, x} \cdots W_{2, +, x} W_{1, +, x} x > 0) W_i.
$$
The matrix $W_{i,+,x}$ contains the rows of $W_i$ that are active after taking a ReLU if  the input to the network is $x$.
Therefore, under the model for $G$, the empirical risk~\eqref{GA:e1} becomes
\begin{equation} \label{GA:e2}
f(x) = \frac{1}{2} \left\| A \left(\prod_{i = d}^1 W_{i, +, x}\right) x - A \left(\prod_{i = d}^1 W_{i, +, x_*}\right) x_* - e \right\|^2.
\end{equation}


\section{Main Results:
Two algorithms and a convergence analysis
} \label{GA:s1}

In this section, we propose two closely related algorithms for minimizing the empirical loss~\eqref{GA:e1}.  The first algorithm is a subgradient descent method which is provably convergent.  The second algorithm is a practical implementation that can be directly implemented with an explicit form of the gradient step that may or may not be within the subdifferential of the objective at some points.

\subsection{A provably convergent subgradient descent method}

In order to state the first algorithm,  Algorithm~\ref{GA:a1}, we first introduce the notion of a subgradient.
Since the cost function $f(x)$ is continuous, piecewise quadratic, and not differentiable everywhere, we use the notion of a generalized gradient, called the Clarke subdifferential or generalized subdifferential \cite{Clason2017}. If a function $\genericf$ is Lipschitz from a Hilbert space $\mathcal{X}$ to $\mathbb{R}$, the Clarke generalized directional derivative of $\genericf$ at the point $x \in \mathcal{X}$ in the direction $u$, denoted by $\genericf^o(x; u)$, is defined by $\genericf^o(x; u) = \limsup_{y \rightarrow x, t \downarrow 0} \frac{\genericf(y + t u) - \genericf(y)}{t}$, and the generalized subdifferential of $\genericf$ at $x$, denoted by $\partial \genericf(x)$, is defined by
\begin{equation*}
\partial \genericf(x) = \{v \in \mathbb{R}^k \mid \inner[]{v}{u} \leq \genericf^o(x; u), \forall u \in \mathcal{X}\}.
\end{equation*}
Any vector in $\partial \genericf(x)$ is called a subgradient of $\genericf$ at $x$. Note that if $\genericf$ is differentiable at $x$, then $\partial \genericf(x) = \{\nabla \genericf(x)\}$.  We can now state Algorithm \ref{GA:a1}.

\begin{algorithm}
\caption{Provably convergent subgradient descent method}
\label{GA:a1}
\begin{algorithmic}[1]
\REQUIRE Weights of the network $W_i$; noisy observation $y$; and step size $\nu > 0$;
\STATE Choose an arbitrary initial point $x_0 \in \mathbb{R}^k \setminus \{0\}$;
\FOR {$i = 0, 1, \ldots$}
\IF {$f(-x_{i}) < f(x_{i})$}  \label{GA:a1:st3}
\STATE $\tilde{x}_{i} \gets - x_{i}$; 
\ELSE
\STATE $\tilde{x}_i \gets x_i$;
\ENDIF \label{GA:a1:st8}
\STATE   Compute \label{GA:a1:st1} $v_{\tilde{x}_i} \in \partial f(\tilde{x}_i)$, in particular, if $G$ is differentiable at $\tilde{x}_i$, then set $v_{\tilde{x}_i} = \tilde{v}_{\tilde{x}_i}$,  where
$$
\tilde{v}_{\tilde{x}_i} := \left(\prod_{i = d}^1 W_{i, +, \tilde{x}_i}\right)^T A^T( A \left(\prod_{i = d}^1 W_{i, +, \tilde{x}_i}\right) \tilde{x}_i - y);
$$
\STATE $x_{i+1} = \tilde{x}_i - \nu {v}_{\tilde{x}_i}$;
\ENDFOR
\end{algorithmic}
\end{algorithm}




This subgradient method has an important twist.  In lines \ref{GA:a1:st3}--\ref{GA:a1:st8}, the algorithm checks whether negating the current iterate of the latent code causes a lower objective, and if so accepts that negation.  The motivation for this step is as follows:  In expectation, the empirical loss $f$ has a global minimum at $x_*$, a local maximum at 0, and a critical point at $-x_* \rho_d$, where $\rho_d \in (0, 1)$,
as established in~ \cite{HV17}.  Moreover, the empirical loss  $f$ concentrates around  its expectation. Thus, a simple gradient descent algorithm could in principle be attracted to $-x_* \rho_d$, and this check is used in order to ensure that it does not.



\subsection{Convergence analysis for Algorithm~\ref{GA:a1}}

In this section, we prove that Algorithm~\ref{GA:a1} converges to the global minimizer $x_*$ up to an error determined by the noise $e$.  Consequently, the signal estimate $G(x_*)$ is also recovered up an error determined by the noise. In the noiseless case (i.e., $e=0$),  Algorithm~\ref{GA:a1} converges to $x^*$, and $G(x^*)$ is recovered exactly.
Our theorem  relies on two deterministic assumptions about the network $G$ and the sensing matrix $A$.

First, we assume that the weights of the network, $W_i$,  satisfy the Weight Distribution Condition (WDC) defined below.  This condition states that the weights are roughly uniformly distributed over a sphere of an appropriate radius.

\begin{definition}[Weight Distribution Condition (WDC)] \label{GA:df1}
A matrix $W \in \mathbb{R}^{n \times k}$ satisfies the \emph{Weight Distribution Condition} with constant $\epsilon$ if for all nonzero $x, y \in \mathbb{R}^k$, it holds that
\begin{equation*}
\left\| \sum_{i = 1}^n 1_{w_i \cdot x > 0} 1_{w_i \cdot y > 0} \cdot w_i w_i^T - Q_{x, y} \right\| \leq \epsilon, \hbox{ with } Q_{x, y} = \frac{\pi - \theta}{2 \pi} I + \frac{\sin \theta}{2 \pi} M_{\hat{x} \leftrightarrow \hat{y}}.
\end{equation*}
Here, $w_i^T \in \mathbb{R}^k$ is the $i$th row of $W$; $M_{\hat{x}} \in \mathbb{R}^{k \times k}$ is the matrix such that $\hat{x} \mapsto \hat{y}$, $\hat{y} \mapsto \hat{x}$, and $\hat{z} \mapsto 0$ for all $z \in \sspan(\{x, y\})^\perp$; $\hat{x} = x / \|x\|$, and $\hat{y} = y / \|y\|$; $\theta = \angle(x, y)$; and $1_S$ is the indicator function on $S$.
\end{definition}

Second, we assume that the measurement matrix $A$ satisfies an isometry condition with respect to $G$, defined below.

\begin{definition}[Range Restricted Isometry Condition (RRIC)] \label{GA:df2}
A matrix $A \in \mathbb{R}^{m \times n}$ satisfies the \emph{Range Restricted Isometry Condition} with respect to $G$ with constant $\epsilon$ if for all $x_1, x_2, x_3, x_4 \in \mathbb{R}^k$, it holds that
\begin{align*}
&\left| \inner[]{A(G(x_1) - G(x_2))}{A(G(x_3) - G(x_4))} - \inner[]{G(x_1) - G(x_2)}{G(x_3) - G(x_4)}\right| \\
\leq& \epsilon \|G(x_1) - G(x_2)\| \|G(x_3) - G(x_4)\|
\end{align*}
\end{definition}


These deterministic conditions are satisfied with some probability by neural networks and measurement matrices that are such that
\begin{enumerate}[label=(\alph*)]
\item \label{GA:as1} The network weights have i.i.d. $\mathcal{N}(0, 1/n_i)$ entries in the $i$th layer.
\item \label{GA:as2} The network is expansive in each layer, in that
\begin{equation}
\label{eq:condition1}
n_i \geq c \epsilon^{-2} \log(1 / \epsilon) n_{i-1} \log n_{i-1},
\end{equation}
where $c$ is a universal constant and $\eps$ is sufficiently small.
\item \label{GA:as3} The measurement vectors have i.i.d. $\mathcal{N}(0, 1/m)$ entries.
\item \label{GA:as4} There are a sufficient number of measurements in that
\begin{equation}
\label{eq:condition2}
m \geq c \epsilon^{-1} \log(1/\epsilon) d k \log \prod_{i = 1}^d n_i,
\end{equation}
where $c$ is a universal constant and $\eps$ is sufficiently small.
\end{enumerate}

The probability that the WDC and RRIC hold with constant $\epsilon$ under the assumption above  is at least
\begin{align}
\label{eq:probestimate}
1 - \sum_{i = 2}^d \tilde{c} n_i e^{- \gamma n_{i-2}} - \tilde{c} n_1 e^{- \gamma \epsilon^2 \log(1/ \epsilon) k} - \tilde{c} e^{- \gamma \epsilon m},
\end{align}
where $\gamma$, and $\tilde{c}$ are universal constants, as given in~\cite[Proposition~4]{HV17}.
%

As our main theoretical result, we prove that the iterates generated by Algorithm~\ref{GA:a1} converge to $x_*$ up to a term dependent on the noise level. The proof is given in Section~\ref{GA:s2}.

\begin{theorem} \label{GA:th1}
Suppose the WDC and RRIC
hold with $\epsilon \leq K_1 / d^{90}$ and the noise $e$ obeys $\|e\| \leq \frac{K_2 \|x_*\|}{d^{42} 2^{d/2}}$. Consider the iterates $\{x_i\}$ generated by Algorithm~\ref{GA:a1} with step size $\nu = K_3 \frac{2^d}{d^2}$. Then there exists a number of iterations, denoted by $N$ and  upper bounded by $N \leq \frac{K_4 f(x_0) 2^d}{d^4 \epsilon \|x_*\|}$ such that
\begin{equation}\label{GA:e63}
\|x_N - x_*\| \leq K_5 d^9 \|x_*\| \sqrt{\epsilon} + K_6 d^6 2^{d/2} \|e\|.
\end{equation}
In addition, for all $i \geq N$, we have
\begin{align}
\|x_{i + 1} - x_*\| &\leq C^{i + 1 - N} \|x_N - x_*\| + K_7 2^{d/2} \|e\| \hbox{ and }  \label{GA:e68} \\
\|G(x_{i + 1}) - G(x_*)\| &\leq \frac{1.2}{2^{d/2}} C^{i + 1 - N} \|x_N - x_*\| + 1.2 K_7 \|e\|, \label{GA:e74}
\end{align}
where $C = 1 - \frac{\nu}{2^d} \frac{7}{8} \in (0, 1)$.
Here, $K_1, K_2, K_3$, $K_4$, $K_5$, $K_6$, and $K_7$ are universal positive constants.
\end{theorem}

Theorem~\ref{GA:th1} shows that after a certain number of iterations $N$, an iterate of Algorithm~\ref{GA:a1} is in a neighborhood of the true latent code $x_*$, and the size of this neighborhood depends on the sample complexity parameter $\epsilon$ and the noise $e$ (see~\eqref{GA:e63}). Furthermore, by~\eqref{GA:e68}, the theorem  guarantees that once the iterates are in this ball, they converge linearly to a smaller neighborhood of $x_*$, and the size of the neighborhood only depends on the noise term~$e$. If the noise term is zero, the algorithm converges linearly to $x_*$. Similarly, it follows from~\eqref{GA:e74} that the recovered image $G(x_i)$ converges to $G(x_*)$ up to the noise.

Note that the factors $2^d$ in the theorem are present because the weights of the coefficients of the matrices $W_i$ have variance $1 / n_i$. As a result, the operation $\relu(W x)$ returns approximately half of the entries of $W x$. Becuase of this, $G(x)$ scales like $2^{d / 2} \|x\|$, the noise $e$ scales like $2^{-d/2}$, and the step size $\nu$ scales like $2^d$. All of these scalings would be unity with an alternate choice of the variance of the entries of $W_i$.

Combining Proposition~4 from the paper~\cite[Proposition~4]{HV17}, with Theorem~\ref{GA:th1} yields the following corollary.

\begin{corollary}
Consider an expansive generative neural network $G$ that satisfies~\ref{GA:as1} and~\ref{GA:as2}, and let the measurements satisfy~\ref{GA:as3} and~\ref{GA:as4}.
Suppose $\epsilon < K_1 / d^{90}$ and
$\|e\| \leq \frac{K_2 \|x_*\|}{d^{42} 2^{d/2}}$.
Then, at least with probability~\eqref{eq:probestimate}, the iterates $\{x_i\}$ generated by Algorithm~\ref{GA:a1} with step size $\nu = K_3 \frac{2^d}{d^2}$ satisfies the following:
There exists a number of step $N$ upper bounded by $N \leq \frac{K_4 f(x_0) 2^d}{d^4 \epsilon \|x_*\|}$ such that
\[
\|x_N - x_*\| \leq K_5 d^9 \|x_*\| \sqrt{\epsilon} + K_6 d^6 2^{d/2} \|e\|.
\]
In addition, for all $i \geq N$, we have
\begin{align*}
\|x_{i + 1} - x_*\| &\leq C^{i + 1 - N} \|x_N - x_*\| + K_7 2^{d/2} \|e\|, \\
\|G(x_{i + 1}) - G(x_*)\| &\leq \frac{1.2}{2^{d/2}} C^{i + 1 - N} \|x_N - x_*\| + 1.2 K_7 \|e\|,
\end{align*}
where $C = 1 - \frac{\nu}{2^d} \frac{7}{8} \in (0, 1)$.
Here, $\gamma$, $\tilde{c}$, $K_1, K_2, K_3$, $K_4$, $K_5$, $K_6$, and $K_7$ are universal positive constants.
\end{corollary}

\subsection{Practical Algorithm}

The empirical risk objective is nondifferentiable on a set of measure zero.  At points of nondifferentiability, Algorithn~\ref{GA:a1} requires selection of a subgradient $\partial f(\tilde{x}_i)$.  Such a subgradient could be determined by computing $\nabla f(\tilde{x}_i + \delta w)$ for a random $w$ and sufficiently small $\delta$.  This is because $f(x)$ is a piecewise quadratic function,  and by~\cite[Theorem~9.6]{Clason2017}, we can express the sub-differential as
\begin{equation} \label{GA:e56}
\partial f(x) = conv(v_1, v_2, \ldots, v_t),
\end{equation}
where $conv$ denotes the convex hull of the vectors $v_1, \ldots, v_t$, $t$ is the number of quadratic functions adjoint to $x$, and $v_i$ is the gradient of the $i$-th quadratic function at $x$.  Because this computation of a subgradient is not explicit, we propose another algorithm, Algorithm~\ref{GA:a3}, where the step direction is simply chosen as ${v}_{\tilde{x}_i} = \tilde{v}_{\tilde{x}_i}$. In practice, it is extremely unlikely to have an iterate on which the function is not differentiable. 
In other words, Algorithm~\ref{GA:a1} reduces in practice to Algorithm~\ref{GA:a3}. However, strictly speaking, the convergence analysis does not apply for Algorithm~\ref{GA:a3} because of the possibility that $ \tilde{v}_{\tilde{x}_i}$ is not a subgradient at $\tilde{x}_i$.
\begin{algorithm}
\caption{Practical gradient descent method}
\label{GA:a3}
\begin{algorithmic}[1]
\REQUIRE Weights of the network $W_i$; noisy observation $y$; and step size $\nu > 0$;
\STATE Choose an arbitrary initial point $x_0 \in \mathbb{R}^k \setminus \{0\}$;
\FOR {$i = 0, 1, \ldots$}
\IF {$f(-x_{i}) < f(x_{i})$}
\STATE $\tilde{x}_{i} \gets - x_{i}$; 
\ELSE
\STATE $\tilde{x}_i \gets x_i$;
\ENDIF 
\STATE   Compute
$
\tilde{v}_{\tilde{x}_i} := \left(\prod_{i = d}^1 W_{i, +, \tilde{x}_i}\right)^T A^T( A \left(\prod_{i = d}^1 W_{i, +, \tilde{x}_i}\right) \tilde{x}_i - y);
$
\STATE $x_{i+1} = \tilde{x}_i - \nu {\tilde v}_{\tilde{x}_i}$;
\ENDFOR
\end{algorithmic}
\end{algorithm}

\section{Experiments}

In this section, we tested the performance of Algorithm~\ref{GA:a3} on synthetic data with various sizes of noise, and verified Theorem~\ref{GA:a3} by numerical results.\footnote{Note that we do not observe that any entry in $W_i W_{i - 1, +, x} \cdots W_{2, +, x} W_{1, +, x} x, \forall i$, is zero in our experiments. Therefore, Algorithm~\ref{GA:a3} is equivalent to Algorithm~\ref{GA:a1} in this case.}

The entries of $A$ are drawn from $\mathcal{N}(0, 1/m)$ and the entries in $W_i$ are drawn from $\mathcal{N}(0, 1 / n_i)$. We consider a two layer network with multiple numbers of input neurons $k$ shown in Figure~\ref{GA:f3}. The numbers of neurons in the middle layer and output layer are fixed to be $250$ and $600$, respectively. The number of rows in the measurement matrix $A$ is $m = 150$. The  latent code $x_*$ and the noise $\tilde{e}$ are drawn from the standard normal distribution. The noisy measurement $y$ is set to be $y = A G(x_*) + \tau \tilde{e} / \|\tilde{e}\|$, and four values of $\tau$ are used such that the signal to noise ratio (SNR) values are 40, 80, 120 and inf, where SNR is defined to be $10 \log_{10} \left(\frac{\|A G(x_*)\|}{\|e\|}\right)$. The step size is chosen to be $2^d / d^2$, which is 1 since $d = 2$.  Algorithm~\ref{GA:a1} stops when either the norm of $\tilde{v}$ is smaller than the machine epsilon or the number of iterations reaches $50000$.

Figure~\ref{GA:f3} reports the empirical probability of successful recovery for noiseless problems. A run is called success if the relative error $\|x - x_*\| / \|x_*\|$ is smaller than $10^{-3}$. We observe that Algorithm~\ref{GA:a1} is able to find the true code $x_*$ when $m$ is sufficiently large relative to $k$.  This experiment shows that signal recovery by empirical risk optimization for compressive sensing under expansive Gaussian generative priors succeeds in a much larger parameter range than that given by the theorem.  In particular, the empirical dependence on $d$ appears to be much milder in practice than what was assumed in the theorem.

\begin{figure}
\begin{center}
\includegraphics[width=1\textwidth]{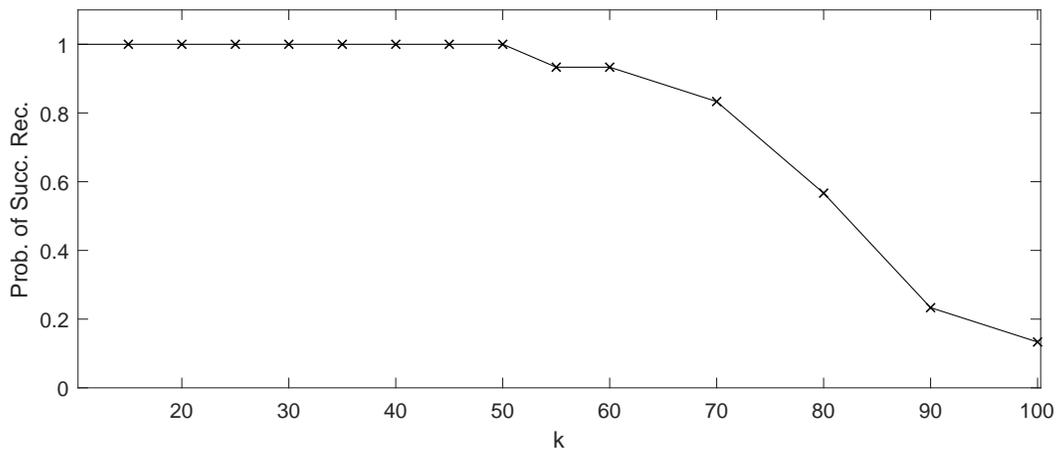}
\caption{Empirical probability of successful recovery (from 30 random runs) versus the number of input neurons $k$ for noiseless problems.  In these experiments, the network has 2 layers, the middle layer has 250 neurons, and the output layer has 600 neurons, after which $m=150$ random measurements are taken. }
\label{GA:f3}
\end{center}
\end{figure}


Figure~\ref{GA:f1} shows graphs of the relative errors versus the number of input neurons at different noise levels. The figure is consistent with the theoretical result in the sense that, fixing $k$, the relative error of the solution found by Algorithm~\ref{GA:a1} is proportional to the norm of the noise, formally stated in~\eqref{GA:e68}. Note that for noisy measurements, the relative error decreases approximately linearly as the number $k$ decreases.  This result is better than what is predicted by the theorem because the theorem is proved in the case of arbitrary noise.  In this case, the noise is random, and one expects superior performance for smaller values of $k$ because only a fraction $k/n$ of the noise energy projects onto the $k$-dimensional signal manifold in $\R^n$. We refer to~\cite{HHHV2018} for more details in the case of random noise.

\begin{figure}
\begin{center}
\includegraphics[width=1\textwidth]{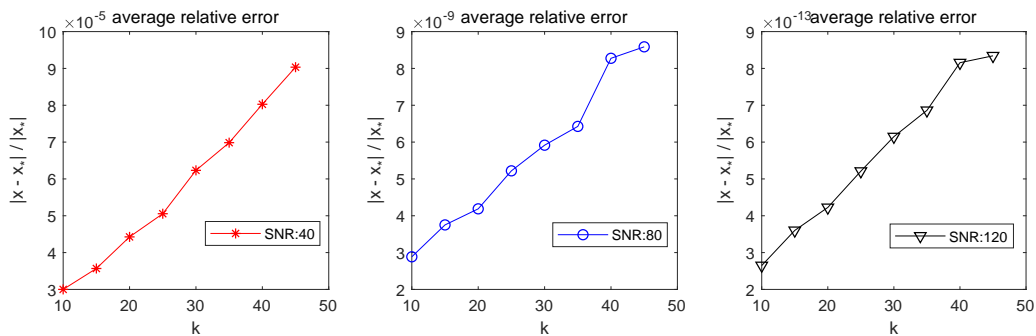}
\caption{The relative error $\|x - x_*\| / \|x_*\|$ versus the number of input neurons $k$. The average of successful runs is reported. 
}
\label{GA:f1}
\end{center}
\end{figure}

Figure~\ref{GA:f2} shows the relationships between the relative error $\|x_i - x_*\| / \|x_*\|$ and the number of iterations for the four values of SNR. The number of input neurons $k$ is $10$. Note that in the tests for the different values of SNR, all the other settings are identical, i.e., the initial iterate, latent code $x_*$, weights matrices $A$ and $W_i$ are the same. We observe that after approximately 20 iterations, Algorithm~\ref{GA:a1} converges linearly to a neighborhood of the true solution $x_*$, and the size of the neighborhood only depends on the magnitude of the noise. These results are consistent with Theorem~\ref{GA:a1}.  Additionally, this figure demonstrates that the relative error in the recovered latent code scales linearly with the magnitude of the noise, which is also consistent with the theorem.

\begin{figure}
\begin{center}
\includegraphics[width=1\textwidth]{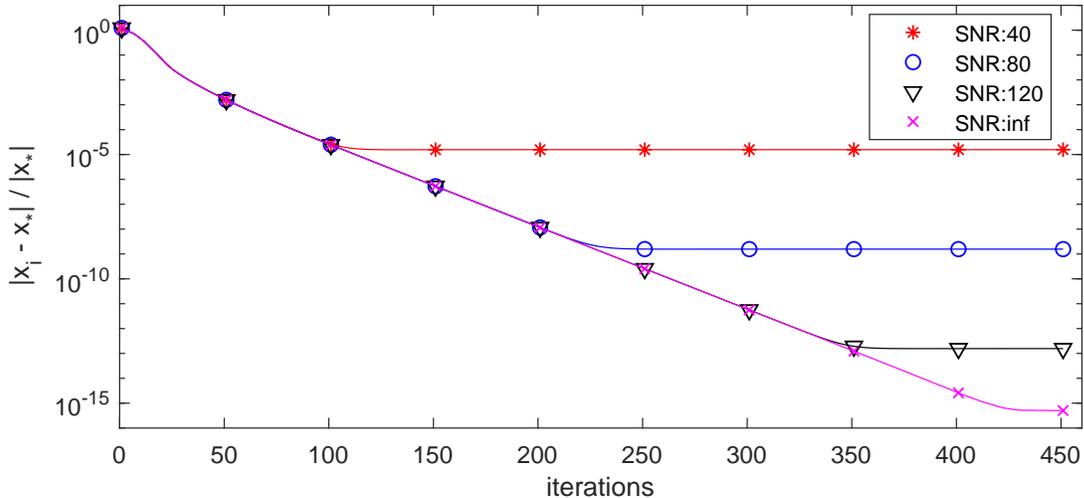}
\caption{The error $\|x - x_*\| / \|x_*\|$ versus the number of iterations. A typical result is reported.}
\label{GA:f2}
\end{center}
\end{figure}

\section{Proof of Theorem~\ref{GA:th1} } \label{GA:s2}

In this section, we prove our main result.
Recall that the goal of Algorithm~\ref{GA:a1} is to minimize the cost function
\begin{equation} \label{GA:e66}
f(x) = \frac{1}{2} \|AG(x) - y\|^2,
\end{equation}
where $y = AG(x_*) + e$ and $e$ is noise.

The proof relies on a concentration of measure argument which ensures that the cost function $f(x)$ and the step direction ${v}_{x}$ concentrate around ${f^E}(x)$ and $h_{x}$, respectively.
In particular, if the WDC and the RRIC hold with $\epsilon = 0$ and the noise $e$ is 0, then $f(x) = {f^E}(x)$ and ${v}_{x} = h_{x}$. The idea of our convergence analysis is to prove properties of ${f^E}(x)$ and the direction $h_{x}$ that are sufficient for a convergence analysis, if our method where to be run on ${f^E}(x)$ with step directions given by $h_x$, and then show that the actual cost function $f(x)$ and step direction are `close enough' to establish convergence.

It is well known that if the gradient of a function is Lipschitz continuous, then a steepest descent method with a sufficient small step size converges to a stationary point from any starting point~\cite{NocWri2006}. However, this result can not be used here since the gradient of the function~\eqref{GA:e2} is not continuous. We overcome this technical difficulty by the following three steps, rigorously stated in Lemma~\ref{GA:le4}, Lemma~\ref{GA:le29}, and Lemma~\ref{GA:le28}, respectively.
\begin{enumerate}
\item\label{it:step1} The function $h_{x}$ is Lipschitz continuous except in a ball around 0.
\item The (sub)-gradient of $f(x)$ is close to $h_x$.
\item\label{it:step3} The iterates generated by Algorithm~\ref{GA:a1} stay sufficiently far away from 0.
\end{enumerate}
Those three steps are sufficient to show that the gradient of $f(x)$ is close to being Lipschitz continuous, and therefore the iterates from Algorithm~\ref{GA:a1} converge to a neighborhood of a \emph{stationary point}.
The size of the neighborhood depends on how close the (sub)-gradient of $f(x)$ is to $h_x$, and is controlled by the noise energy $\|e\|$ and the variable $\epsilon$ in the WDC and the RRIC.
Of course, we also have to ensure that the algorithm not only converges to any of the three stationary points, but that it actually converges to a  point close to $x_*$, for this we rely on the `tweak'  of the algorithm in steps~\ref{it:step1}-\ref{it:step3}.

The remainder of the proof is organized as follows.
We start by defining notation used throughout the proof (see Section~\ref{GA:s5}).
In Section~\ref{sec:prelims} we introduced several technical results formalizing the steps 1-3 above, and in Section~\ref{GA:s4} we use those properties to formally prove  Theorem~\ref{GA:th1}.

\subsection{Notation} \label{GA:s5}

Here, we define some useful quantities, in particular ${f^E}(x)$ and $h_x$, and introduce standard notation used throughout.
We start with defining a function that is helpful for controlling how the operator $x \to W_{+,x}x$ distorts angles, and is defined as
\[
g(\theta) = \cos^{-1}\left(\frac{1}{\pi}\left[ (\pi - \theta) \cos \theta + \sin \theta \right]\right).
\]
With this notation, define
\begin{align*}
h_{x, y} =& \frac{1}{2^d} x - \tilde{h}_{x, y},
\end{align*}
where
\[
\tilde{h}_{x, y} = \frac{1}{2^d} \left( \prod_{i = 0}^{d - 1} \frac{\pi - \bar{\theta}_{i, x, y}}{\pi} \right) y + \frac{1}{2^d} \sum_{i = 0}^{d - 1} \frac{\sin \bar{\theta}_{i, x, y}}{\pi} \left( \prod_{j = i + 1}^{d-1} \frac{\pi - \bar{\theta}_{j, x, y}}{\pi} \right) \|y\| \hat{x}.
\]
Here, $\bar{\theta}_{0, x, y} =  \angle(x, y)$ and  $\bar{\theta}_{i, x, y} = g(\bar{\theta}_{i - 1, x, y})$.
Moreover, given a vector $z \in \mathbb{R}^t$, $\hat{z} = {z} / \|{z}\|$.
For simplicity of notation, we use $h_{z}$ and $\tilde{h}_z$ to denote $h_{z, x_*}$ and $\tilde{h}_{z, x_*}$, respectively, i.e., we omit $x_*$.
Next, define
\begin{align*}
{f^E}(x) =& \frac{1}{2^{d+1}} x^T x - x^T \tilde{h}_{x, x_*} + \frac{1}{2^{d+1}} x_*^T x_*.
\end{align*}
Moreover, let
\[
\rho_d =  \sum_{i = 0}^{d - 1} \frac{\sin \check{\theta}_i}{\pi} \left(\prod_{j = i + 1}^{d - 1} \frac{\pi - \check{\theta}_j}{\pi}\right),
\]
where $\check\theta_0 = \pi$, and $\check{\theta}_i = g(\check{\theta}_{i-1})$. 

Next, define $\mathcal{B}(x, a) \defeq \{y \in \mathbb{R}^k \mid \|y - x\| \leq a\}$,
let $u = v + c O_1(t)$ denote $\|u - v\| \leq c |t|$,
and $f(t) = O(g(t))$ denotes $\lim_{t \rightarrow \infty} |f(t)| / |g(t)| < C$, where $C>0$ is a constant.
Moreover, $\|\cdot\|$ denotes the spectral norm.

Finally, define
\begin{align*}
S_{\epsilon} \defeq&  \{x \in \mathbb{R}^k \mid \|h_{x, x_*}\| \leq \frac{1}{2^d} \epsilon \max(\|x\|, \|x_*\|)\}, \\
S_{\epsilon}^+ \defeq&  S_{\epsilon} \cap \mathcal{B}(x_*, 5000 d^6 \epsilon \|x_*\|  ),  \hbox{ and } \\
S_{\epsilon}^- \defeq&  S_{\epsilon} \cap \mathcal{B}( - \rho_d x_*, 500 d^{11} \sqrt{\epsilon} \|x_*\|). \\
\end{align*}

%
%
%

\subsection{Preliminaries \label{sec:prelims}}

In this section, we state formal results making steps~\ref{it:step1}-\ref{it:step3} from the beginning of this section rigorous, and collect properties used later in the  proof of  Theorem~\ref{GA:th1}.

We  start by showing that the function $h_x$ is Lipschitz continuous except in a ball around $0$:

\begin{lemma} \label{GA:le4}
For all $x, y \neq 0$, it holds that
\begin{equation*}
\|h_{x} - h_{y}\| \leq \left(\frac{1}{2^d} + \frac{6 d + 4 d^2}{\pi 2^d} \max\left( \frac{1}{\|x\|}, \frac{1}{\|y\|} \right) \|x_*\|\right) \|x - y\|.
\end{equation*}
In addition, if $x, y \notin \mathcal{B}(0, r \|x_*\|)$ for any $r > 0$, then $\|h_{x} - h_{y}\| \leq \left(\frac{1}{2^d} + \frac{6 d + 4 d^2}{\pi r 2^d}\right) \|x - y\|$.
\end{lemma}


The next lemma states that $h_x$ and the sub-gradient $v_x$ are close:

\begin{lemma} \label{GA:le29}
Suppose the WDC and RRIC
hold with $\epsilon \leq 1 / (16 \pi d^2)^2$. Then for any $x\neq 0$ and any  $v_x \in {\partial f(x)}$,
\[
\|v_x - h_x\| \leq a_1 \frac{d^3 \sqrt{\eps}}{2^d} \max ( \|x\|, \| \xo\|) + \frac{2}{2^{d/2}} \|e\|,
\]
where $a_1$ is a universal constant.
\end{lemma}

Next, we ensure that after sufficiently many steps, the algorithm will be relatively far from the maximum around $0$:

\begin{lemma} \label{GA:le28}
Suppose that WDC holds with $\epsilon < 1/ (16 \pi d^2)^2$ and $\|e\| \leq \frac{\|x_*\|}{8\pi 2^{d/2}}$.
Moreover, suppose that the step size in Algorithm~\ref{GA:a1} satisfies $0< \nu  < \frac{a_2 2^d}{d^2}$, where $a_2$ is a universal constant.
Then, after at most $T = (\frac{19 \pi 2^d}{16 \nu})^2$ steps, we have that for all $i>T$
and for all $t \in [0, 1]$ that
$t \tilde{x}_i + (1 - t) x_{i + 1} \notin \mathcal B(0, \frac{1}{32\pi} \|x_*\|)$.
\end{lemma}

Recall that $S_\beta$ is the set of points with the norm of $h_x$ upper bounded by $\beta \max(\|x\|, \|x_*\|)$.
The following lemma shows that this set is contained in balls around $x_*$ and $\rho_d x_*$, thus outside those balls, the norm of $h_x$ is lower bounded, which, together with Lemma~\ref{GA:le29}, establishes that the sub-gradients are bounded away from zero. This in turn is important to show that outside those balls our gradient scheme makes progress.
\begin{lemma}{\cite[Lemma~8]{HHHV2018}} \label{lemma:Sepsst}
For any $\beta \leq \frac{1}{64^2 d^{12}}$,
\begin{align*}
S_\beta
\subset
\mathcal{B}(x_*, 5000 d^6 \beta \|x_*\|  ) \cup  \mathcal{B}( - \rho_d x_*, 500 d^{11} \sqrt{\beta} \|x_*\|).
\end{align*}
Here, $\rho_d>0$ obeys $\rho_d \to 1$ as $d \to \infty$.
\end{lemma}

It has been shown in~\cite{HV17} that the function ${f^E}(x)$ has three stationary points: one at $- \rho_d x_*$, one global minimizer at $x_*$ and a local maximizer at $0$. Therefore, Algorithm~\ref{GA:a1} could in principle be attracted to $- \rho_d x_*$. Lemma~\ref{GA:le14} guarantees that with the tweak from Step~\ref{GA:a1:st3} to Step~\ref{GA:a1:st8}, the iterates of Algorithm~\ref{GA:a1} converges to a neighborhood of $x_*$.
\begin{lemma} \label{GA:le14}
Suppose the WDC and RRIC
hold with $\epsilon < 1 / (16 \pi d^2)^2$. Moreover, suppose the noise $e$ satisfies $\|e\| \leq \frac{a_3 \|x_*\|}{d^2 2^{d/2}}$, where $a_3$ is a universal constant. Then for any $\phi_d \in [\rho_d, 1]$, it holds that
\begin{align} \label{GA:e25}
f(x) < f(y)
\end{align}
for all $x \in \mathcal{B}(\phi_d x_*, a_4 d^{-10} \|x_*\|)$ and $y \in \mathcal{B}(- \phi_d x_*, a_4 d^{-10} \|x_*\|)$, where $a_4 < 1$ is a universal constant.
\end{lemma}

Once an iterate is in a small neighborhood of $x_*$, Lemma~\ref{GA:le10} guarantees that the search directions of the iterates afterward point to $x_*$ up to the noise $e$. Therefore, the iterates by Algorithm~\ref{GA:a1} converge to $x_*$ up to the noise. In other words, the parameter $\epsilon$ in WDC and RRIC does not influence the size of the neighborhood that iterates converge to.

\begin{lemma} \label{GA:le10}
Suppose the WDC and RRIC
hold with $200 d \sqrt{d \sqrt{\epsilon}} < 1$ and $x \in \mathcal{B}(x_*, d \sqrt{\epsilon} \|x_*\|)$.  Then
for all $x \neq 0$ and for all $v_x \in \partial f(x)$,
\begin{equation*}
\left\|v_x - \frac{1}{2^d} (x - x_*)\right\| \leq \frac{1}{2^d} \frac{1}{8} \|x - x_*\| + \frac{2}{2^{d/2}} \|e\|.
\end{equation*}
\end{lemma}

\subsection{Proof of Theorem~\ref{GA:th1}} \label{GA:s4}

The proof can be divided into three parts. We first show that the iterates $\{x_i\}$ converge to a neighborhoods of $x_*$ and $-\rho_d x_*$, whose sizes depend on $\epsilon$ and the noise energy $\|e\|$.
Second, we show that the iterates only converge to the neighborhood of $x_*$ that depends both on $\epsilon$ as well as on the noise energy $\|e\|$.
Lastly, we show that once an iterate is in the aforementioned neighborhood of $x_*$, the subsequent iterates converge to a neighborhood of $x_*$ whose size only depends on the noise $\|e\|$, but not on $\epsilon$.

\begin{enumerate}
\item {\bf Convergence to a neighborhood of $x_*$ or $- \rho_d x_*$:}
We prove that if $\|h_{x_i}\|$ is sufficiently large, specifically if the iterate $x_i$ is not in the set $S_{\beta}$ with
$$
\beta = 4 a_1 d^3 \sqrt{\epsilon} + 26 \|e\| 2^{d/2} / \|x_*\|,
$$
then Algorithm~\ref{GA:a1} makes progress in the sense that $f(x_{i + 1}) - f(x_i)$ is smaller than a certain negative value. Therefore, the iterates of Algorithm~\ref{GA:a1} converge to $S_\beta$. 

Consider $i$ such that $\tilde{x}_i \notin S_\beta$.
Let $\eta_{\hat{x}_i} \in \partial f(\hat{x}_i)$ and define $\hat{x}_i = \tilde{x}_i - a \nu {v}_{\tilde{x}_i}$, where $a \in [0, 1]$. By the mean value  theorem, for some  $a \in [0, 1]$, we have
\begin{align}
f(\tilde{x}_i - \nu {v}_{\tilde{x}_i}) - f(\tilde{x}_i) =& \inner[]{\eta_{\hat{x}_i}}{- \nu {v}_{\tilde{x}_i}} 
\nonumber \\
=& \inner[]{\tilde{v}_{\tilde{x}_i}}{- \nu {v}_{\tilde{x}_i}} + \inner[]{\eta_{\hat{x}_i} - {v}_{\tilde{x}_i}}{- \nu {v}_{\tilde{x}_i}} \nonumber \\
\leq& - \nu \|{v}_{\tilde{x}_i}\| ( \|{v}_{\tilde{x}_i}\| - \|\eta_{\hat{x}_i} - {v}_{\tilde{x}_i}\| ). \label{GA:e21}
\end{align}

Next, we provide a lower and upper bound of the terms $\|{v}_{\tilde{x}_i}\|$ and $\|v_{\hat{x}_{i}} - {v}_{\tilde{x}_i}\|$, respectively, which appear on the right hand side of~\eqref{GA:e21}.

First, we have
\begin{align}
\|{v}_{\tilde{x}_i}\| \geq& \|h_{\tilde{x}_i}\| - \|h_{\tilde{x}_i} - {v}_{\tilde{x}_i}\| \nonumber \\
\geq& 2^{-d} \max(\|\tilde{x}_i\|, \|x_*\|) \left( \beta - a_1 d^3 \sqrt{\eps}  - 2 \|e\| \frac{2^{d/2}}{\|x_*\|} \right) \nonumber \\
\geq& 2^{-d} \max(\|\tilde{x}_i\|, \|x_*\|) \left( 3 a_1 d^3 \sqrt{\eps}  + 24 \|e\| \frac{2^{d/2}}{\|x_*\|} \right) \label{GA:le30} \\
\geq& 2^{-d} \|x_*\| 3 a_1 d^3 \sqrt{\epsilon}, \label{GA:e71}
\end{align}
where the second inequality follows from the definition of $S_\beta$ and Lemma~\ref{GA:le29}, 
and the third inequality follows from the definition of $\beta$.


Second, by Lemma~\ref{GA:le4} and Lemma~\ref{GA:le28}, for all $a \in [0, 1]$ and $i > T$ ($T$ is defined in Lemma~\ref{GA:le28}), we have
\begin{equation} \label{lem:hxhatmhlip}
\|h_{\hat{x}_{i}} - h_{\tilde{x}_{i}}\| \leq \frac{b_0 d^2}{2^d} \|\hat{x}_{i} - \tilde{x}_{i}\|,
\end{equation}
where $\hat{x}_{i} = \tilde{x}_{i} - a \nu {v}_{\tilde{x}_i}$, and $b_0$ is a universal constant.
Thus, for any $v_{\hat{x}_{i}} \in \partial f(\hat{x}_{i}) $,
\begin{align}
\hspace{0.4cm}&\hspace{-0.4cm}\|v_{\hat{x}_{i}} - {v}_{\tilde{x}_i}\|
\leq
\|v_{\hat{x}_{i}} - h_{\hat{x}_{i}}\| + \|h_{\hat{x}_{i}} - h_{\tilde{x}_{i}}\| + \|h_{\tilde{x}_{i}} - {v}_{\tilde{x}_i}\| \nonumber \\
&\leq
a_1 \frac{d^3 \sqrt{\epsilon}}{2^d} \max(\|\hat{x}_{i}\|, \|x_*\|) + \frac{2}{2^{d/2}} \|e\| + \frac{b_0 d^2}{2^d} \|\hat{x}_{i} - \tilde{x}_{i}\| \nonumber \\
&+ a_1 \frac{d^3 \sqrt{\epsilon}}{2^d} \max(\|\tilde{x}_{i}\|, \|x_*\|) + \frac{2}{2^{d/2}} \|e\| \nonumber \\
&\leq
a_1 \frac{d^3 \sqrt{\epsilon}}{2^d} \max(\|\tilde{x}_{i}\| + \nu \|{v}_{\tilde{x}_i}\|, \|x_*\|) + \frac{b_0 d^2}{2^d} \nu \|{v}_{\tilde{x}_i}\| \nonumber \\
&+ a_1 \frac{d^3 \sqrt{\epsilon}}{2^d} \max(\|\tilde{x}_{i}\|, \|x_*\|) + \frac{4}{2^{d/2}} \|e\|  \nonumber \\
&\leq
a_1 \frac{d^3 \sqrt{\epsilon}}{2^d} \left(2 + \frac{\nu d a_6}{2^d}  \right) \max(\|\tilde{x}_{i}\|, \|x_*\|) + \frac{b_0 d^2}{2^d} \nu \|{v}_{\tilde{x}_i}\| + 4 \frac{K_2/d^{42}}{2^{d}} \|x_*\|, \label{eta-minus-vtilde}
\end{align}
where the second inequality follows from Lemma~\ref{GA:le29} and~\eqref{lem:hxhatmhlip},
and the fourth inequality follows from Lemma~\ref{GA:le32}
and the assumption $\|e\| \leq \frac{K_2 \|x_*\|}{d^{42} 2^{d/2}}$.

Combining~\eqref{eta-minus-vtilde} and~\eqref{GA:le30}, we get that
\begin{align*}
\|v_{\hat{x}_{i}} - {v}_{\tilde{x}_{i}}\|
\leq&
\left(
\frac{5}{6}  + \nu b_1 \frac{d^2}{2^d}
\right)
\|{v}_{\tilde{x}_i} \|,
\end{align*}
with the appropriate constants chosen sufficiently small, where $b_1$ is a universal constant. Choosing $v_{\hat{x}_i} = \eta_{\hat{x}_{i}}$ yields
\begin{equation} \label{GA:e62}
\|\eta_{\hat{x}_{i}} - {v}_{\tilde{x}_i}\| \leq \biggl(\frac{5}{6}  + \nu b_0 \frac{d^2}{2^d}\biggr) \| {v}_{\tilde{x}_i} \|.
\end{equation}
Therefore, combining~\eqref{GA:e21} and~\eqref{GA:e62} yields
\begin{equation} \label{GA:e72}
f(\tilde{x}_i - \nu {v}_{\tilde{x}_i}) - f(\tilde{x}_i) \leq - \frac{1}{12} \nu \|{v}_{\tilde{x}_i}\|^2,
\end{equation}
where we used that $\nu b_0 \frac{d^2}{2^d} \leq 1 / 12$ by the assumption that the step size obeys $\nu = K_3 2^d / d^2$ and by
taking $K_3$ appropriately small.
Applying~\eqref{GA:e71} to \eqref{GA:e72} yields
\begin{equation*}
f(\tilde{x}_i - \nu {v}_{\tilde{x}_i}) - f(\tilde{x}_i) \leq - \frac{1}{12} \nu \|{v}_{\tilde{x}_i}\|^2 \leq - 2^{-d} d^4 b_1 \epsilon \|x_*\|^2,
\end{equation*}
where $b_1$ is a universal constant and we used $\nu = K_3 \frac{2^d}{d^2}$. Therefore, there can be at most $\frac{f(x_0) 2^d}{b_1 d^4 \epsilon \|x_*\|^2}$ iterations for which $\tilde{x}_i \notin S_\beta$. In other words, there exists $N \leq \frac{f(x_0) 2^d}{b_1 d^4 \epsilon \|x_*\|^2}$ such that $\tilde{x}_N \in S_\beta$.

\item {\bf Convergence to a neighborhood of $x_*$:}
Note that by the assumption $\|e\| \leq \frac{K_2 \|x_*\| }{d^{42} 2^{d/2} }$ and $\epsilon \leq K_1 / d^{90}$,
our choice of $\beta$ obeys $\beta \leq \frac{1}{64^2 d^{12}}$ for sufficiently small $K_1, K_2$, and thus the assumptions of Lemma~\ref{lemma:Sepsst} are met and we have
\begin{equation} \label{GA:e69}
S_\beta
\subset
\mathcal{B}(x_*, r)
\cup
\mathcal{B}( - \rho_d x_*, \sqrt{r \|x_*\|} d^8).
\end{equation}
Here, we defined the radius
$r = K_{5} d^9 \sqrt{\epsilon} \|x_*\| + K_{6} d^6 \|e\| 2^{d/2}$, and $K_{5}$ and $K_{6}$ are universal constants and are used in~\eqref{GA:e63}.

By the assumption $\|e\| \leq \frac{K_2 \|x_*\| }{d^{42} 2^{d/2} }$ and $\epsilon \leq K_1 / d^{90}$ and choosing $K_1$ and $K_2$ sufficiently small, we have $r \leq a_4 d^{-10} \|x_*\|$ and $\sqrt{r \|x_*\|} d^8 \leq a_4 d^{-10} \|x_*\|$.
Note that the powers of $d$ in the upper bounds of $\|e\|$ and $\epsilon$, which are $-42$ and $-90$ respectively, are used to get $\sqrt{r \|x_*\|} d^8 \leq a_4 d^{-10} \|x_*\|$. It follows from~\eqref{GA:e69} that
$$
S_\beta^+ \subset \mathcal{B}(x_*, a_4 d^{-10} \|x_*\|) \hbox{ and } S_\beta^- \subset \mathcal{B}(- \rho_d x_*, a_4 d^{-10} \|x_*\|).
$$
Therefore, by Lemma~\ref{GA:le14}, for any $x \in S_\beta^-$ and $y \in S_\beta^+$, it holds that $f(x) > f(y)$.
Thus, if $\tilde{x}_N \in S_\beta$, then $\tilde{x}_N$ must be in $S_\beta^+$ due to the operations from Step~\ref{GA:a1:st3} to Step~\ref{GA:a1:st8}. 

We claim that if $x_i$ is inside the ball $\mathcal{B}(x_*, r)$, then all iterates afterward stay in $\mathcal{B}(x_*, 2 r)$. To see this, note that by Lemma~\ref{GA:le32} and the choice of the step size, we have for any $v_{\tilde{x}_i} \in {\partial f({\tilde{x}_i})}$,
$\nu \|v_{\tilde{x}_i}\| \leq \frac{a_6}{d 2^d} \max(\|x\|, \|x_*\|)$.

\item {\bf Convergence to $x_*$ up to the noise $e$:}
Next we show that for any $i \geq N$, it holds that $x_i \in \mathcal{B}(x_*, a_4 d^{-10} \|x_*\|)$, $\tilde{x}_i = x_i$, and
\begin{align*}
\|x_{i + 1} - x_*\| \leq b_2^{i + 1 - N} \|x_N - x_*\| + b_4 2^{d/2} \|e\|.
\end{align*}
where $a_4$ is defined in Lemma~\ref{GA:le14}, $b_2 = 1 - \frac{\nu}{2^d} \frac{7}{8}$ and $b_4$ is a universal constant.

Suppose $\tilde{x}_i \in \mathcal{B}(x_*, a_4 d^{-10} \|x_*\|)$. By the assumption $\epsilon \leq K_1 / d^{90}$ for sufficiently small $K_1$, the assumptions in Lemma~\ref{GA:le10} are met. Therefore,
\begin{align}
\|x_{i + 1} - x_*\|
=& \|\tilde{x}_i - \nu {v}_{\tilde{x}_i} - x_*\|   \nonumber \\
=& \|\tilde{x}_i - x_*  -  \frac{\nu}{2^d} (\tilde{x}_i - x_*) - \nu {v}_{\tilde{x}_i} + \frac{\nu}{2^d} (\tilde{x}_i - x_*) \| \nonumber \\
\leq& \left( 1 - \frac{\nu}{2^d} \right) \|\tilde{x}_i - x_*\| + \nu \|{v}_{\tilde{x}_i} - \frac{1}{2^d} (\tilde{x}_i - x_*)\| \nonumber \\
\leq& \left( 1 - \frac{\nu}{2^d} \right) \|\tilde{x}_i - x_*\| + \nu \left( \frac{1}{8} \frac{1}{2^d} \|\tilde{x}_i - x_*\| + \frac{2}{2^{d/2}} \|e\| \right) \nonumber \\
=& \left( 1 - \frac{\nu}{2^d} \frac{7}{8} \right) \|\tilde{x}_i - x_*\| + \nu \frac{2}{2^{d/2}} \|e\|, \label{GA:e64}
\end{align}
where the second inequality holds by Lemma~\ref{GA:le10}.
By the assumptions $\tilde{x}_i \in \mathcal{B}(x_*, a_4 d^{-10} \|x_*\|)$, $\|e\| \leq \frac{K_2 \|x_*\| }{d^{42} 2^{d/2} }$, and using~\eqref{GA:e64}, we have $x_{i + 1} \in \mathcal{B}(x_*, a_4 d^{-10} \|x_*\|)$. In addition, using Lemma~\ref{GA:le14} yields that $\tilde{x}_{i + 1} = x_{i+1}$. Repeat the above steps yields that $x_i \in \mathcal{B}(x_*, a_4 d^{-10} \|x_*\|)$ and $\tilde{x}_i = x_i$ for all $ i \geq N$.

Using~\eqref{GA:e64} and $\nu = K_3 \frac{2^d}{d^2}$, we have
\begin{align} \label{GA:e65}
\|x_{i + 1} - x_*\|\leq b_2 \|x_i - x_*\| + b_3 \frac{2^{d/2}}{d^2} \|e\|,
\end{align}
where $b_2 = 1 - 7 K_3 / (8 d^2)$ and $b_3$ is a universal constant. Repeatedly applying~\eqref{GA:e65} yields
\begin{align*}
\|x_{i + 1} - x_*\| \leq& b_2^{i + 1 - N} \|x_N - x_*\| + (b_2^{i-N} + b_2^{i - N - 1} + \cdots + 1) \frac{b_3 2^{d/2}}{d^2} \|e\| \\
\leq& b_2^{i + 1 - N} \|x_N - x_*\| + \frac{b_3 2^{d/2}}{(1 - b_2) d^2} \|e\| \\
\leq& b_2^{i + 1 - N} \|x_N - x_*\| + b_4 2^{d/2} \|e\|,
\end{align*}
where the last inequality follows from the definition of $b_2$ and the step size $\nu = K_3 \frac{2^d}{d^2}$, and $b_4$ is a universal constant. This finishes the proof for~\eqref{GA:e68}. Inequality~\eqref{GA:e74} follows from Lemma~\ref{GA:le19}.
\end{enumerate}

This concludes the proof of  our main result.
In the remainder, we provide proofs of  the lemmas above.

\subsection{Proof of Lemma~\ref{GA:le4}}

It holds that
\begin{align}
&\|x - y\| \geq 2 \sin (\theta_{x, y} / 2) \min(\|x\|, \|y\|), &\forall x, y \label{GA:e4} \\
&\sin(\theta / 2) \geq \theta / 4, &\forall \theta \in [0, \pi] \label{GA:e5} \\
&\frac{d}{d \theta} g(\theta) \in [0, 1] &\forall \theta \in [0, \pi] \label{GA:e13} 
\end{align}
where $\theta_{x, y} = \angle(x, y)$.

For brevity of notation, let $\zeta_{j, z} = \prod_{i = j}^{d - 1} \frac{\pi - \bar{\theta}_{i, z, x_*}}{\pi}$.
Combining \eqref{GA:e4} and~\eqref{GA:e5} gives $|\bar{\theta}_{0, x, x_*} - \bar{\theta}_{0, y, x_*}| \leq 4 \max\left( \frac{1}{\|x\|}, \frac{1}{\|y\|} \right) \|x - y\|$.
Inequality~\eqref{GA:e13} implies $|\bar{\theta}_{i, x, x_*} - \bar{\theta}_{i, y, x_*}| \leq |\bar{\theta}_{j, x, x_*} - \bar{\theta}_{j, y, x_*}|, \forall i \geq j$.
It follows that
\begin{align}
\|h_{x, x_*} - h_{y, x_*}\| \leq& \frac{1}{2^d}\|x - y\| + \frac{1}{2^d} \underbrace{\left| \zeta_{0, x} - \zeta_{0, y} \right|}_{T_1} \|x_*\| \nonumber \\
+& \frac{1}{2^d} \underbrace{ \left| \sum_{i = 0}^{d - 1} \frac{\sin \bar{\theta}_{i, x, x_*}}{\pi} \zeta_{i + 1, x} \hat{x} - \sum_{i = 0}^{d - 1} \frac{\sin \bar{\theta}_{i, y, x_*}}{\pi} \zeta_{i + 1, y} \hat{y} \right| }_{T_2} \|x_*\|. \label{GA:e50}
\end{align}
We use the following result which is proven later in Lemma~\ref{GA:le17}:
\begin{align} \label{GA:e53}
T_1 \leq \frac{d}{\pi} |\bar{\theta}_{0, x, x_*} - \bar{\theta}_{0, y, x_*}| \leq \frac{4 d}{\pi}\max\left( \frac{1}{\|x\|}, \frac{1}{\|y\|} \right) \|x - y\|.
\end{align}
Additionally, it holds that
\begin{align}
T_2 =& \left| \sum_{i = 0}^{d - 1} \frac{\sin \bar{\theta}_{i, x, x_*}}{\pi} \zeta_{i + 1, x} \hat{x} - \frac{\sin \bar{\theta}_{i, x, x_*}}{\pi} \zeta_{i + 1, x} \hat{y} + \frac{\sin \bar{\theta}_{i, x, x_*}}{\pi} \zeta_{i + 1, x} \hat{y} - \sum_{i = 0}^{d - 1} \frac{\sin \bar{\theta}_{i, y, x_*}}{\pi} \zeta_{i + 1, y} \hat{y} \right| \nonumber \\
\leq& \frac{d}{\pi} \|\hat{x} - \hat{y}\| + \underbrace{ \left| \sum_{i = 0}^{d - 1} \frac{\sin \bar{\theta}_{i, x, x_*}}{\pi} \zeta_{i + 1, x} - \sum_{i = 0}^{d - 1} \frac{\sin \bar{\theta}_{i, y, x_*}}{\pi} \zeta_{i + 1, y} \right|}_{T_3} . \label{GA:e51}
\end{align}
We have
\begin{align}
T_3 \leq& \sum_{i = 0}^{d - 1} \left[ \left| \frac{\sin \bar{\theta}_{i, x, x_*}}{\pi} \zeta_{i + 1, x} -\frac{\sin \bar{\theta}_{i, x, x_*}}{\pi} \zeta_{i + 1, y} \right|\right. + \left.\left| \frac{\sin \bar{\theta}_{i, x, x_*}}{\pi} \zeta_{i + 1, y} - \frac{\sin \bar{\theta}_{i, y, x_*}}{\pi} \zeta_{i + 1, y} \right| \right] \nonumber \\
\leq& \sum_{i = 0}^{d - 1} \left[ \frac{1}{\pi} \left( \frac{d - i  - 1}{\pi} \left|\bar{\theta}_{i-1, x, x_*} - \bar{\theta}_{i-1, y, x_*}\right| \right) + \frac{1}{\pi} |\sin \bar{\theta}_{i, x, x_*} - \sin \bar{\theta}_{i, y, x_*}|  \right] \nonumber \\
\leq& \frac{d^2}{\pi} |\bar{\theta}_{0, x, x_*} - \bar{\theta}_{0, y, x_*}| \leq \frac{4 d^2}{\pi} \max\left(\frac{1}{\|x\|}, \frac{1}{\|y\|}\right) \|x - y\|. \label{GA:e52}
\end{align}
Using~\eqref{GA:e4} and~\eqref{GA:e5} and noting $\|\hat{x} - \hat{y}\| \leq \theta_{x, y, x_*}$ yield
\begin{equation} \label{GA:e15}
\|\hat{x} - \hat{y}\| \leq \theta_{x, y, x_*} \leq 2 \max\left( \frac{1}{\|x\|}, \frac{1}{\|y\|} \right) \|x - y\|.
\end{equation}
Finally, combining~\eqref{GA:e50}, \eqref{GA:e53}, \eqref{GA:e51}, \eqref{GA:e52} and~\eqref{GA:e15} yields the result.

\subsection{Proof of Lemma~\ref{GA:le29}}

Let $\bar{v}_x = \left(\prod_{i = d}^1 W_{i, +, x}\right)^T A^T (A \left(\prod_{i = d}^1 W_{i, +, x}\right) x - A \left(\prod_{i = d}^1 W_{i, +, x_*}\right) x_*) $ and $q_x = \left(\prod_{i = d}^1 W_{i, +, x}\right)^T A^T e$. Therefore, $\tilde{v}_x = \bar{v}_x - q_x$.

For any $x \neq 0$ and suppose $G(x)$ is differentiable at $x$, we have
\begin{align}
\|\tilde{v}_x - h_x\|
&=
\|\bar{v}_x + q_x- h_x\| \leq
\|\bar{v}_x-h_x\| + \|q_x\| \nonumber \\
&\leq
b_0 \frac{d^3 \sqrt{\eps}}{2^d} \max ( \|x\|, \| x_* \|)
+
\|q_x\| \nonumber \\
&\leq
b_0 \frac{d^3 \sqrt{\eps}}{2^d} \max ( \|x\|, \| x_* \|)
+
\frac{2}{2^{d/2}} \|e\|, \label{GA:e70}
\end{align}
where the second inequality follows from~\cite[(29)]{HV17} and the third inequality follows from Lemma~\ref{GA:le31} given later.

Since $f(x)$ is a piecewise quadratic function, by~\cite[Theorem~9.6]{Clason2017}, we have
\begin{equation*}
\partial f(x) = conv(v_1, v_2, \ldots, v_t),
\end{equation*}
where $conv$ denotes the convex hull of the vectors $v_1, \ldots, v_t$, $t$ is the number of quadratic functions adjoint to $x$ and $v_i$ is the gradient of the $i$-th quadratic function at $x$. Therefore, for any $v \in \partial f(x)$, there exist $c_1, c_2, \ldots, c_t \geq 0$ such that $c_1 + c_2 + \ldots + c_t = 1$ and $v = c_1 v_1 + c_2 v_2 + \ldots + c_t v_t$. Note that for any $v_i$, there exists $u_i$ so that $v_i = \lim_{\delta \rightarrow 0^+} \nabla f(x + \delta_i u_i)$, and $f$ is differentiable at $f(x + \delta u_i)$ for sufficiently small $\delta$.

The proof is concluded by appealing to the continuity of $h_x$ with respect to nonzero $x$, inequality~\eqref{GA:e70}, and by  noting that
\begin{align*}
\|v_x - h_x\|
&\leq \sum_i c_i \|v_i - h_x\|
= \sum_i c_i \|\lim_{\delta_i \rightarrow 0^+} \nabla f(x + \delta_i u_i) - h_x\| \\
&= \sum_i c_i \lim_{\delta_i \rightarrow 0^+} \|\nabla f(x + \delta_i u_i) - h_{x + \delta_i u_i}\| = \sum_i c_i \lim_{\delta_i \rightarrow 0^+} \|\tilde{v}_{x + \delta_i u_i} - h_{x + \delta_i u_i}\| \\
&\leq b_0 \frac{d^3 \sqrt{\eps}}{2^d} \max ( \|x\|, \| x_* \|)
+
\frac{2}{2^{d/2}} \|e\|,
\end{align*}
where we used the inequality above and that $\sum_i c_i = 1$.

\subsection{Proof of Lemma~\ref{GA:le28}}

First suppose that $\tilde{x}_i \in \mathcal{B}(0,\frac{1}{16\pi} \|x_*\|)$.
We show that after a polynomial number of iterations $N$, we have that $\tilde{x}_{i+N} \notin \mathcal{B}(0,\frac{1}{16\pi} \|x_*\|)$.
Below, we use that
\begin{align}
\label{eq:properstepzero}
\text{
$\inner[]{x}{{v}_{x}} < 0$
and
$\|{v}_{x}\| \geq \frac{1}{2^d 16\pi} \|x_*\|$
for all
$x \in \mathcal{B}(0, \frac{1}{16\pi} \|x_*\|)$} \hbox{ and } v_x \in {\partial f(x)},
\end{align}
which will be proven later.
It follows that for any
$\tilde{x}_i \in \mathcal{B}(0, \frac{1}{16\pi} \|x_*\|)$,
$\tilde{x}_i$ and the next iterate produced by the algorithm,
$x_{i+1} = \tilde{x}_i - \nu {v}_{\tilde{x}_i}$, and the origin form an obtuse triangle.
As a consequence,
\begin{align}
\|\tilde{x}_{i + 1}\|^2 = \|x_{i + 1}\|^2
&\geq
\|\tilde{x}_i\|^2 + \nu^2 \|{v}_{\tilde{x}_i}\|^2 \nonumber \\
&\geq
\|\tilde{x}_i\|^2 + \nu^2 \frac{1}{(2^{d} 16\pi)^2} \|x_*\|^2, \label{GA:e73}
\end{align}
where the last inequality follows from~\eqref{eq:properstepzero}.
Thus, the norm of the iterates $\tilde{x}_i$ will increase until after $\bigl(\frac{2^d}{\nu} \bigr)^2$  iterations, we have $\tilde{x}_{i+N} \notin \mathcal{B}(0, \frac{1}{16\pi} \|x_*\|)$.

Consider $\tilde{x}_i \notin \mathcal{B}(0, \frac{1}{16\pi} \|x_*\|)$, and note that
\begin{align*}
\nu \|{v}_{\tilde{x}_i}\| \leq \nu \frac{d a_6}{2^d} \max(\|\tilde{x}_i\|, \|x_*\|) \leq \nu\frac{16 \pi a_6 d}{2^d}\|\tilde{x}_i\| \leq \frac{1}{2}\|\tilde{x}_i\|,
\end{align*}
where the first inequality follows from Lemma~\ref{GA:le32}, the second inequality from $\|\tilde{x}_i\| \geq \frac{1}{16\pi} \|x_*\|$, and finally the last inequality from our assumption on the step size $\nu$. Therefore, from $x_{i + 1} = \tilde{x}_i - \nu {v}_{\tilde{x}_i}$, we have that $t \tilde{x}_i + (1 - t) x_{i+1} \notin \mathcal{B}(0, \frac{1}{32 \pi} \|x_*\|)$ for all $t \in [0, 1]$, which completes the proof.

It remains to prove~\eqref{eq:properstepzero}.
We start with proving $\inner[]{x}{\tilde{v}_x} < 0$.
For brevity of notation, let $\Lambda_{z} = \prod_{i = d}^1 W_{i, +, z}$. We have
We have
\begin{align*}
x^T \tilde{v}_x =&
\inner[]{ \Lambda_x^T A^T A \Lambda_x x - \Lambda_x^T A^T A \Lambda_{x_*} x_* + \Lambda_x^T A^T e }{x} \\
\leq& \inner[]{ \Lambda_x^T A^T A \Lambda_x x - \Lambda_x^T A^T A \Lambda_{x_*} x_* - \Lambda_x^T \Lambda_x x + \Lambda_x^T \Lambda_{x_*} x_* }{x} \\
&+ \inner[]{ \Lambda_x^T \Lambda_x x - \Lambda_x^T \Lambda_{x_*} x_* + \Lambda_x^T A^T e }{x}  \\
\leq& \epsilon \|\Lambda_x x\|^2 + \epsilon \|\Lambda_x x\| \|\Lambda_{x_*} x_*\| + \inner[]{ \Lambda_x^T \Lambda_x x - \Lambda_x^T \Lambda_{x_*} x_* + \Lambda_x^T A^T e }{x}  \\
\leq&
\frac{13}{12}2^{-d} \|x\|^2 - \frac{1}{4 \pi} \frac{1}{2^d} \|x\| \|x_*\| +
\|x\| \frac{2}{2^{d/2}} \|e\| \\
\leq&
\|x\| \left( \frac{13}{12}2^{-d} \|x\| +
\frac{1/(8\pi)}{2^{d}} \|x_*\|
- \frac{1}{4 \pi} \frac{1}{2^d} \|x_*\| \right) \\
\leq&
\|x\| \frac{1}{2^d} \left( 2\|x\| - \frac{1}{8 \pi} \|x_*\| \right).
\end{align*}
The second inequality follows from RRIC, ~\cite[(10)]{HV17} that $\|\Lambda_x x\|^2 \leq \frac{1+4\epsilon d}{2^d} \leq \frac{13}{12} \frac{1}{2^d}$; the third inequality follows from~\cite[Lemma~6]{HV17} that $\inner[]{\Lambda_x x}{\Lambda_{x_*} x_*} \geq \frac{1}{4 \pi} \frac{1}{2^d} \|x\| \|x_*\|$, and the fourth inequality follows from Lemma~\ref{GA:le31}.
Therefore, for any $x \in \mathcal{B}(0, \frac{1}{16\pi} \|x_*\|)$, $\inner[]{x}{ \tilde{v}_{x} } < 0$,  as desired.

If $G(x)$ is differentiable at $x$, then $v_x = \tilde{v}_x$ and $\inner[]{x}{{v}_x} < 0$. If $G(x)$ is not differentiable at $x$, by equation~\eqref{GA:e56}, we have
\begin{align*}
x^T {v}_x =& x^T (c_1 v_1 + c_2 v_2 + \cdots + c_t v_t) \leq (c_1 + c_2 + \ldots + c_t) \|x\| \frac{1}{2^d} \left( 2\|x\| - \frac{1}{8 \pi} \|x_*\| \right) \\
=& \|x\| \frac{1}{2^d} \left( 2\|x\| - \frac{1}{8 \pi} \|x_*\| \right),
\end{align*}
for all $v_x \in \partial f(x)$.

We next show that, for any
$x \in \mathcal{B}(0, \frac{1}{16\pi} \|x_*\|)$, and $v_x \in \partial f(x)$, it holds that $\|{v}_{x}\| \geq \frac{1}{2^d 16\pi} \|x_*\|$.

If $G(x)$ is differentiable at $x$, then
the local linearity of $G$ gives that $G(x + z) - G(x) = \Lambda_{x} z$ for any sufficiently small $z \in \mathbb{R}^k$. Using the RRIC and \cite[(10)]{HV17}, we have
\begin{align*}
&|\langle A \Lambda_{x} z, A \Lambda_{x} x - A \Lambda_{x_*} x_*\rangle - \inner[]{\Lambda_{x} z}{\Lambda_{x} x - \Lambda_{x_*} x_*} | \\
\leq& \epsilon \left\| \Lambda_{x} z \right\| \left\| \Lambda_{x} x \right\| + \left\|\Lambda_{x_*} x_* \right\| \leq \epsilon \frac{1}{2^{d}} (1 + 2 \epsilon d)^2 (\|x\| + \|x_*\|) \|z\|.
\end{align*}
Therefore, $\left\|\Lambda_{x}^T A^T (A\Lambda_{x} x - A\Lambda_{x_*} x_* ) - \Lambda_{x}^T (\Lambda_{x} x - \Lambda_{x_*} x_* )\right\| \leq \epsilon \frac{1}{2^{d}} (1 + 2 \epsilon d)^2 (\|x\| + \|x_*\|)$. Therefore, we have
\begin{align*}
\|{v}_x\| =& \|\Lambda_x^T A^T (A \Lambda_x x -  A \Lambda_{x_*} x_*) - \Lambda_x^T (\Lambda_x x - \Lambda_{x_*} x_*) + \Lambda_x^T \Lambda_x x - \Lambda_x^T \Lambda_{x_*} x_* + \Lambda_x^T e\| \\
\geq& \|\Lambda_x^T \Lambda_{x_*} x_*\| - \|\Lambda_x^T \Lambda_x x\| - \|\Lambda_x^T A^T (A \Lambda_x x -  A \Lambda_{x_*} x_*) - \Lambda_x^T (\Lambda_x x - \Lambda_{x_*} x_*) \| \\
&- \epsilon \frac{1}{2^{d}} (1 + 2 \epsilon d)^2 (\|x\| + \|x_*\|) -
\|\Lambda_x^T e\| \\
\geq&
\frac{1}{4 \pi} \frac{1}{2^d} \|x_*\| - \frac{13}{12}\frac{1}{2^d} \|x\|- \epsilon \frac{1}{2^{d}} (1 + 2 \epsilon d)^2 (\|x\| + \|x_*\|) -\frac{2}{2^{d/2}} \|e\|  
\\
\geq&
\frac{1}{2^d} \left( \frac{1}{8 \pi} - \frac{1}{16\pi} \right) \|x_*\|.
\end{align*}
If $G(x)$ is not differentiable at $x$, by equation~\eqref{GA:e56}, we have
\begin{align*}
\|v_x\| =& \|c_1 v_1 + c_2 v_2 + \cdots + c_t v_t\| \leq c_1\|v_1\| + c_2\|v_2\| + \ldots + c_t\|v_t\| = \frac{1}{2^d} \left( \frac{1}{8 \pi} - \frac{1}{16\pi} \right) \|x_*\|,
\end{align*}
for all $v_x \in \partial f(x)$.
This concludes the proof of~\eqref{eq:properstepzero}.

\subsection{Proof of Lemma~\ref{GA:le14}}

Consider the function
\[
f_\eta(x) = f_0(x) - \langle A G(x) - A G(x_*), e \rangle,
\]
and note that $f(x) = f_\eta(x) + \|e\|^2$.
Consider $x \in \mathcal{B}(\phi_d x_*, \varphi \|x_*\|)$, for a $\varphi$ that will be specified later.
Note that
\begin{align*}
\left|\inner[]{A G(x) - A G(x_*)}{e}\right|
&\leq
|\inner[]{A \prod_{i = d}^1 W_{i, +, x} x}{e}|
+
|\inner[]{A \prod_{i = d}^1 W_{i, +, x_*} x_*}{e}| \\
&=
|\inner[]{x}{(\prod_{i = d}^1 W_{i, +, x})^T A^T e}|
+
|\inner[]{x_*}{(\prod_{i = d}^1 W_{i, +, x_*})^T A^T e}| \\
&\leq (\|x\| + \|x_*\|) \frac{2}{2^{d/2}}\|e\| \\
&\leq (\varphi \|x_*\| + \|x_*\|) \frac{2}{2^{d/2}} \|e\|,
\end{align*}
where the second inequality holds by Lemma~\ref{GA:le31}, and the last inequality holds by our assumption on $x$.
Thus, by Lemma~\ref{GA:le13} and Lemma~\ref{GA:le15}, we have
\begin{align}
f_\eta(x)
\leq&
f^E_0(x) + |f_0(x) - f^E_0(x)| +
\left|\inner[]{A G(x) - A G(x_*)}{e}\right|
\nonumber \\
\leq&
\frac{1}{2^{d+1}} \left( \phi_d^2 - 2 \phi_d + \frac{10}{a_8^3} d \varphi \right) \|x_*\|^2 + \frac{1}{2^{d+1}} \|x_*\|^2 \nonumber \\
&+
\frac{\epsilon (1 + 4 \epsilon d)}{2^d} \|x\|^2 + \frac{\epsilon(1 + 4 \epsilon d) + 48 d^3 \sqrt{\epsilon}}{2^{d+1}} \|x\| \|x_*\| + \frac{\epsilon (1 + 4 \epsilon d)}{2^d} \|x_*\|^2 \nonumber \\
&+ (\varphi \|x_*\| + \|x_*\|) \frac{2}{2^{d/2}} \|e\| \label{GA:e67}
\end{align}
Additionally, for $x \in \mathcal{B}(\phi_d x_*, \varphi  \|x_*\|)$, we have
\begin{align}
\eqref{GA:e67} \leq&
\frac{1}{2^{d+1}} \left( \phi_d^2 - 2 \phi_d + \frac{10}{a_8^3} d \varphi \right) \|x_*\|^2 + \frac{1}{2^{d+1}} \|x_*\|^2 \nonumber \\
&+
\frac{\epsilon (1 + 4 \epsilon d)}{2^d} (\phi_d + \varphi)^2 \|x_*\|^2 + \frac{\epsilon(1 + 4 \epsilon d) + 48 d^3 \sqrt{\epsilon}}{2^{d+1}} (\phi_d + \varphi)\|x_*\|^2 + \frac{\epsilon (1 + 4 \epsilon d)}{2^d} \|x_*\|^2  \nonumber \\
&+
(\varphi \|x_*\| + \|x_*\|) \frac{2}{2^{d/2}} \|e\| \nonumber \\
\leq&
\frac{ \|x_*\|^2 }{2^{d+1}}
\left(
1+
\phi_d^2 - 2 \phi_d + \frac{10}{a_8^3} d \epsilon
+
68 d^2 \sqrt{\epsilon}
\right)
+
(\varphi \|x_*\| + \|x_*\|) \frac{2}{2^{d/2}} \|e\| \label{GA:e42}
\end{align}
where the last inequality follows from $\epsilon < \sqrt{\epsilon}$, $\rho_d \leq 1$, $4 \epsilon d < 1$, $\varphi < 1$ and assuming $\varphi = \epsilon$.

Similarly, we have that for any $y \in \mathcal{B}(- \phi_d x_*, \varphi  \|x_*\|)$
\begin{align}
f_\eta(y) \geq& \mathbb{E}[f(y)] - |f(y) - \mathbb{E}[f(y)]|
-
\left|\inner[]{A G(x) - A G(x_\ast)}{e}\right|
  \nonumber \\
\geq& \frac{1}{2^{d+1}} \left(\phi_d^2 - 2 \phi_d \rho_d - 10 d^3 \varphi \right) \|x_*\|^2 + \frac{1}{2^{d+1}} \|x_*\|^2 \nonumber \\
&- \left( \frac{\epsilon (1 + 4 \epsilon d)}{2^d} \|y\|^2 + \frac{\epsilon(1 + 4 \epsilon d) + 48 d^3 \sqrt{\epsilon}}{2^{d+1}} \|y\| \|x_*\| + \frac{\epsilon (1 + 4 \epsilon d)}{2^d} \|x_*\|^2 \right) \nonumber \\
&- (\varphi \|x_*\| + \|x_*\|) \frac{2}{2^{d/2}} \|e\| \nonumber \\
\geq&
\frac{\|x_*\|^2}{2^{d+1}} \left(1 + \phi_d^2 - 2 \phi_d \rho_d - 10 d^3 \varphi
- 68 d^2 \sqrt{\epsilon}
\right) -(\varphi \|x_*\| + \|x_*\|) \frac{2}{2^{d/2}} \|e\|
\label{GA:e43}
\end{align}
Using $\epsilon < \sqrt{\epsilon}$, $\rho_d \leq 1$, $4 \epsilon d < 1$, $\varphi < 1$, $\|e\| \leq \frac{K_2\|x_*\|}{d^{42} 2^{d/2}} \leq \frac{K_2\|x_*\|}{d^{2} 2^{d/2}}$ and assuming $\varphi = \epsilon$, the right side of~\eqref{GA:e42} is smaller than the right side of~\eqref{GA:e43} if
\begin{equation} \label{GA:e44}
\varphi = \epsilon \leq \left(\frac{(1  - \rho_d) \phi_d  - 4 K_2/d^2}{\left( 125 + \frac{5}{a_8^3} \right) d^3} \right)^2.
\end{equation}
It follows from Lemma~\ref{lemma:rho-d} that $1 - \rho_d \geq 1 / (a_7(d+2)^2)$.
Thus, it suffices to have $\varphi = \epsilon = \frac{a_4}{d^{10}}$ and $4 K_2 / d^2 \leq \frac{1}{2} \frac{1}{a_7 (d+2)^2} \leq 1 - \rho_d$ for an appropriate universal constant $K_2$, and for an appropriate universal constant $a_4$.  

\subsection{Proof of Lemma~\ref{GA:le10}}

For brevity of notation, let $\Lambda_{j, z} = \prod_{i = j}^1 W_{i, +, z}$.
Suppose the function $G(x)$ is differentiable at $x$. Then the local linearity of $G$ gives that $G(x + z) - G(x) = \Lambda_{j, x} z$ for any sufficiently small $z \in \mathbb{R}^k$. Using the RRIC, \cite[(10)]{HV17} and Lemma~\ref{GA:le19}, we have
\begin{align*}
&|\langle A \Lambda_{j, x} z, A \Lambda_{j, x} x - A \Lambda_{j, x_*} x_*\rangle - \inner[]{\Lambda_{j, x} z}{\Lambda_{j, x} x - \Lambda_{j, x_*} x_*} | \\
\leq& \epsilon \left\| \Lambda_{j, x} z \right\| \left\| \Lambda_{j, x} x - \Lambda_{j, x_*} x_* \right\| \leq \epsilon \frac{1}{2^{\frac{d}{2}}} (1 + 2 \epsilon d) \left\| \Lambda_{j, x} x - \Lambda_{j, x_*} x_* \right\| \|z\|\\
\leq& \epsilon \frac{1.2}{2^{d}} (1 + 2 \epsilon d) \|x - x_*\| \|z\|.
\end{align*}
Therefore, $\left\|\bar{v}_{x} - \Lambda_{j, x}^T (\Lambda_{j, x} x - \Lambda_{j, x_*} x_* )\right\| \leq \epsilon \frac{1.2}{2^{d}} (1 + 2 \epsilon d) \|x - x_*\| \leq \frac{1}{16} \frac{1}{2^d} \|x - x_*\|$. Combining with Lemma~\ref{GA:le8} yields that
$$
\|\bar{v}_{x} - \frac{1}{2^d} (x - x_*)\| \leq \frac{1}{2^d} \frac{1}{8} \|x - x_*\|.
$$
It follows that
$$
\|\tilde{v}_x - \frac{1}{2^d} (x - x_*)\| = \|\bar{v}_x + q_x - \frac{1}{2^d} (x - x_*)\| \leq \frac{1}{2^d} \frac{1}{8} \|x - x_*\| + \frac{2}{2^{d/2}} \|e\|.
$$

For any $x \neq 0$ and for any $v \in \partial f(x)$, by~\eqref{GA:e56}, there exist $c_1, c_2, \ldots, c_t \geq 0$ such that $c_1 + c_2 + \ldots + c_t = 1$ and $v = c_1 v_1 + c_2 v_2 + \ldots + c_t v_t$. It follows that
$\|v - \frac{1}{2^d} (x - x_*)\| \leq \sum_{j = 1}^t c_j \|v_j - \frac{1}{2^d} (x - x_*)\| \leq \frac{1}{2^d} \frac{1}{8} \|x - x_*\| + \frac{2}{2^{d/2}} \|e\|$.


\bibliographystyle{alpha}


\bibliography{WHlibrary}


\appendix

\section{Supporting Lemmas}

Lemma~\ref{GA:le32} is used in proofs for Section~\ref{GA:s4} and Lemma~\ref{GA:le28}.
\begin{lemma} \label{GA:le32}
Suppose that the WDC and RRIC holds with $\epsilon < 1/(16 \pi d^2)^2$ and that the noise $e$ satisfies $\|e\| \leq a_5 2^{-d/2} \|x_*\|$. Then, for all $x$ and all $v_x \in \partial f(x)$,
\begin{align} \label{GA:e23}
\|{v}_x\|
\leq
\frac{a_6 d}{2^d} \max(\|x\|, \|x_*\|),
\end{align}
where $a_5$ and $a_6$ are universal constants.
\end{lemma}
\begin{proof}
Define for convenience $\zeta_j=\prod_{i = j}^{d - 1} \frac{\pi - \bar{\theta}_{j, x, x_*}}{\pi}$.
We have
\begin{align}
\|{v}_x\|
\leq&
\|h_{x}\| + \|h_{x} - {v}_x\| \nonumber \\
\leq&
\left\|\frac{1}{2^d} x - \frac{1}{2^d} \zeta_{0} x_* - \frac{1}{2^d} \sum_{i = 0}^{d - 1} \frac{\sin \bar{\theta}_{i,x}}{\pi} \zeta_{i + 1}  \frac{\|x_*\|}{\|x\|} x \right\|
+
a_1 \frac{d^3 \sqrt{\eps}}{2^d} \max ( \|x\|, \| x_* \|)
+
\frac{2}{2^{d/2}} \|e\|
\nonumber \\
\leq&
\frac{1}{2^d} \|x\| + \left( \frac{1}{2^d}  + \frac{d}{\pi 2^d} \right) \|x_*\| + a_1 \frac{d^3 \sqrt{\epsilon}}{2^d} \max(\|x\|, \|x_*\|) + \frac{2}{2^{d/2}} \|e\| \nonumber \\
\leq&
\frac{a_6 d}{2^d} \max(\|x\|, \|x_*\|), \nonumber
\end{align}
where the second inequality follows from the definition of $h_x$ and Lemma~\ref{GA:le29}, the third inequality uses $| \zeta_j | \leq 1$,
and the last inequality uses the assumption $\|e\| \leq a_5 2^{-d/2} \|x_*\|$.
\end{proof}

Lemma~\ref{GA:le17} is used in proofs for Lemma~\ref{GA:le4}.
\begin{lemma} \label{GA:le17}
Suppose $a_i, b_i \in [0, \pi]$ for $i = 1, \ldots, k$, and $|a_i - b_i| \leq |a_j - b_j|, \forall i \geq j$. Then it holds that
\begin{equation*}
\left|\prod_{i = 1}^k \frac{\pi - a_i}{\pi} - \prod_{i = 1}^k \frac{\pi - b_i}{\pi}\right| \leq \frac{k}{\pi} |a_1 - b_1|.
\end{equation*}
\end{lemma}
\begin{proof}
Prove by induction. It is easy to verify that the inequality holds if $k = 1$.
Suppose the inequality holds with $k = t - 1$. Then
\begin{align*}
\left|\prod_{i = 1}^{t} \frac{\pi - a_i}{\pi} - \prod_{i = 1}^t \frac{\pi - b_i}{\pi}\right| \leq& \left|\prod_{i = 1}^{t} \frac{\pi - a_i}{\pi} - \frac{\pi - a_t}{\pi} \prod_{i = 1}^{t-1} \frac{\pi - b_i}{\pi}\right| \\
&+ \left|\frac{\pi - a_t}{\pi} \prod_{i = 1}^{t-1} \frac{\pi - b_i}{\pi} - \prod_{i = 1}^t \frac{\pi - b_i}{\pi}\right| \\
\leq& \frac{t-1}{\pi} |a_1 - b_1| + \frac{1}{\pi} |a_t - b_t| \leq \frac{t}{\pi} |a_1 - b_1|.
\end{align*}
\end{proof}

Lemma~\ref{GA:le31} is used in proofs for Lemma~\ref{GA:le29}, Lemma~\ref{GA:le28}, and Lemma~\ref{GA:le14}.
\begin{lemma} \label{GA:le31}
Suppose the WDC and RRIC
hold with $\epsilon \leq 1 / (16 \pi d^2)^2$. Then we have
$$
\left|x^T q_x\right| \leq \frac{2}{2^{d/2}} \|e\| \|x\|,
$$
where $q_x = \left(\prod_{i = d}^1 W_{i, +, x}\right)^T A^T e$. In addition, if $x$ is differentiable at $G(x)$, then we have
\begin{equation*}
\left\|q_x\right\| \leq \frac{2}{2^{d/2}} \|e\|.
\end{equation*}
\end{lemma}
\begin{proof}
We have
\begin{align*}
|x^T q_x|^2 =& |e^T A G(x)|^2 \leq \|A G(x)\|^2 \|e\|^2 \leq (1 + \epsilon) \|G(x)\|^2 \|e\|^2 \\
\leq& (1 + \epsilon) \prod_{i = d}^1\| W_{i, +, x} \|^2 \|e\|^2 \|x\|^2 \leq (1 + \epsilon) (1 + 2 \epsilon d)^2 \frac{1}{ 2^{d}} \|e\|^2 \|x\|^2
\end{align*}
where the second inequality follows from RRIC and the last inequality follows from~\cite[(10)]{HV17}. Therefore, $\left|x^T q_x\right| \leq \frac{2}{2^{d/2}} \|e\| \|x\|$.


Suppose $G$ is differentiable at $x$. Then the local linearity of $G$ implies that $G(x + z) - G(x) = \left(\prod_{i = d}^1 W_{i, +, x}\right) z$ for any sufficiently small $z \in \mathbb{R}^k$. By the RRIC, we have
\begin{align*}
&\left| \inner[]{ A \left(\prod_{i = d}^1 W_{i, +, x}\right) z }{A \left(\prod_{i = d}^1 W_{i, +, x}\right) z} - \inner[]{ \left(\prod_{i = d}^1 W_{i, +, x}\right) z }{\left(\prod_{i = d}^1 W_{i, +, x}\right) z} \right| \\
&\leq \epsilon \prod_{i = d}^1\| W_{i, +, x} \|^2 \|z\|^2,
\end{align*}
which implies
\begin{equation*}
\left| \inner[]{ A \left(\prod_{i = d}^1 W_{i, +, x}\right) z }{A \left(\prod_{i = d}^1 W_{i, +, x}\right) z} \right| \leq (1 + \epsilon) \prod_{i = d}^1\| W_{i, +, x} \|^2 \|z\|^2.
\end{equation*}
Therefore, we obtain
\begin{equation*}
\left\| A \left(\prod_{i = d}^1 W_{i, +, x}\right) \right\| \leq \sqrt{1 + \epsilon} \prod_{i = d}^1\| W_{i, +, x} \|.
\end{equation*}
Combining above inequality with $\prod_{i = d}^1\| W_{i, +, x} \| \leq (1 + 2 \epsilon d) / 2^{d/2} \leq 1.5 / 2^{d/2}$ given in~\cite[(10)]{HV17} yields
\begin{equation*}
\left\| A \left(\prod_{i = d}^1 W_{i, +, x}\right) \right\| \leq 1.5 \sqrt{1 + \epsilon} / 2^{d/2} \leq 2 / 2^{d/2},
\end{equation*}
where the second inequality follows from the assumption on $\epsilon$.
Therefore, we obtain
\begin{equation*}
\|q_x\| = \left\| \left(\prod_{i = d}^1 W_{i, +, x}\right)^T A^T e\right\| \leq \left\| \left(\prod_{i = d}^1 W_{i, +, x}\right)^T A^T\right\| \|e\| \leq \frac{2}{2^{d/2}} \|e\|.
\end{equation*}
\end{proof}

Lemma~\ref{lemma:rho-d} is used in proofs for Lemma~\ref{GA:le14}.
\begin{lemma}\label{lemma:rho-d}
For all $d\geq 2$, that
\[
1/\left(a_7(d + 2)^2\right) \leq 1 - \rho_d \leq 250/(d + 1),
\]
\end{lemma}
and $a_8 = \min_{d \geq 2} \rho_d > 0$.
\begin{proof}

It holds that
\begin{align}
&\log(1+x) \leq x &\forall x \in [-0.5, 1] \label{GA:e47} \\
&\log(1-x) \geq -2 x &\forall x \in [0, 0.75] \label{GA:e48} 
\end{align}
where $\theta_{x, y} = \angle(x, y)$.

We recall the results in \cite[(36), (37), and (50)]{HV17}:
\begin{align*}
&\check{\theta}_i \leq \frac{3 \pi}{i + 3}\;\;\;\;\;  \;\;\; \hbox{ and }\;\;\;\;\;  \;\;\; \check{\theta}_i \geq \frac{\pi}{i + 1}\;\;\;\;\; \forall i \geq 0 \\
&1 - \rho_d = \prod_{i = 1}^{d - 1} \left( 1 - \frac{\check{\theta}_{i}}{\pi} \right) + \sum_{i = 1}^{d-1} \frac{\check{\theta}_{i} - \sin \check{\theta}_{i}}{\pi} \prod_{j = i+1}^{d-1} \left( 1 - \frac{\check{\theta}_{j}}{\pi} \right).
\end{align*}
Therefore, we have for all $0 \leq i \leq d - 2$,
\begin{align*}
\prod_{j = i+1}^{d-1} \left( 1 - \frac{\check{\theta}_{j}}{\pi} \right) \leq& \prod_{j = i+1}^{d-1} \left( 1 - \frac{1}{j + 1} \right) = e^{\sum_{j = i + 1}^{d - 1} \log\left(1 - \frac{1}{j + 1}\right)} \\
&\leq e^{- \sum_{j = i + 1}^{d - 1}  \frac{1}{j + 1}} \leq e^{- \int_{i + 1}^d \frac{1}{s + 1} d s} = \frac{i + 2}{d + 1}, \\
\prod_{j = i+1}^{d-1} \left( 1 - \frac{\check{\theta}_{j}}{\pi} \right) \geq& \prod_{j = i+1}^{d-1} \left( 1 - \frac{3}{j + 3} \right) = e^{\sum_{j = i + 1}^{d - 1} \log\left(1 - \frac{3}{j + 3}\right)} \\
&\geq e^{- \sum_{j = i + 1}^{d - 1}  \frac{6}{j + 3}} \geq e^{- \int_{i}^{d - 1} \frac{6}{s + 3} d s} = \left(\frac{i + 3}{d + 2}\right)^6,
\end{align*}
where the second and the fifth inequalities follow from~\eqref{GA:e47} and~\eqref{GA:e48} respectively.
Since $\pi^3 / (12 (i + 1)^3) \leq \check{\theta}_{i}^3 / 12 \leq \check{\theta}_{i} - \sin \check{\theta}_{i} \leq \check{\theta}_{i}^3 / 6 \leq 27 \pi^3 / (6 (i + 3)^3)$, we have that for all $d \geq 3$
\begin{align*}
1 - \rho_d \leq& \frac{2}{d + 1} + \sum_{i = 1}^{d - 1} \frac{27 \pi^3}{6 (i + 3)^3} \frac{i + 2}{d + 1} \leq \frac{2}{d + 1} + \frac{3 \pi^5}{4 (d + 1)} \leq \frac{250}{d + 1}
\end{align*}
and
\begin{align*}
1 - \rho_d \geq& \left(\frac{3}{(d+2)}\right)^6 + \sum_{i = 1}^{d - 1} \frac{\pi^3}{12 (i + 3)^3} \left( \frac{i + 3}{d + 2} \right)^6 \geq \frac{1}{K_1 (d + 2)^2},
\end{align*}
where we use $\sum_{i = 4}^\infty \frac{1}{i^2} \leq \frac{\pi^2}{6}$ and $\sum_{i = 1}^n i^3 = O(n^4)$.
Since $\rho_d \geq 1 - 250 / (d+1)$ and $\rho_d > 0$  for all $d \geq 2$, we have $\min_{d \geq 2} \rho_d > 0$.
\end{proof}

Lemma~\ref{GA:le13} is used in proofs for Lemma~\ref{GA:le14}.

\begin{lemma} \label{GA:le13}
Fix $0 < a_9 < \frac{1}{4 d^2 \pi}$. For any $\phi_d \in [\rho_d, 1]$, it holds that
\begin{align*}
{f^E}(x) <& \frac{1}{2^{d+1}} \left( \phi_d^2 - 2 \phi_d + \frac{10}{a_8^3} d a_9 \right) \|x_*\|^2 + \frac{\|x_*\|^2}{2^{d+1}}, \forall x \in \mathcal{B}(\phi_d x_*, a_9  \|x_*\|) \hbox{ and } \\
{f^E}(x) >& \frac{1}{2^{d+1}} \left(\phi_d^2 - 2 \phi_d \rho_d - 10 d^3 a_9 \right) \|x_*\|^2 + \frac{\|x_*\|^2}{2^{d+1}}, \forall x \in \mathcal{B}(- \phi_d x_*, a_9  \|x_*\|),
\end{align*}
where $a_8$ is defined in Lemma~\ref{lemma:rho-d}.
\end{lemma}
\begin{proof}
If $x \in \mathcal{B}(\phi_d x_*, a_9 \|x_*\|)$, then we have
$0 \leq \bar{\theta}_{0, x, x_*} \leq \arcsin(a_9 / \phi_d) \leq \frac{\pi a_9}{2 \phi_d}$,
$0 \leq \bar{\theta}_{0, x, x_*} \leq \bar{\theta}_{i, x, x_*} \leq \frac{\pi a_9}{2 \phi_d}$, and
$\phi_d \|x_*\| - a_9 \|x_*\| \leq \|x\| \leq \phi_d \|x_*\| + a_9 \|x_*\|$.
Note that $\cos \theta \geq 1 - \frac{\theta^2}{2}, \forall \theta \in [0, \pi]$.
We have
\begin{align*}
&{f^E}(x) - \frac{\|x_*\|^2}{2^{d+1}} \leq \frac{1}{2^{d+1}} \|x\|^2 - \frac{1}{2^d} \left( \prod_{i = 0}^{d - 1} \frac{\pi - \bar{\theta}_{i, x, x_*}}{\pi} \right) x_*^T x \\
\leq& \frac{1}{2^{d+1}} (\phi_d + a_9)^2 \|x_*\|^2 - \frac{1}{2^d} \left( \prod_{i = 0}^{d - 1} \frac{\pi - \frac{\pi a_9}{2 \phi_d}}{\pi} \right) \|x_*\| \|x\| \cos \bar{\theta}_{0, x, x_*} \\
\leq& \frac{1}{2^{d+1}} (\phi_d + a_9)^2 \|x_*\|^2 - \frac{1}{2^d} \left( \prod_{i = 0}^{d - 1} \frac{\pi - \frac{\pi a_9}{2 \phi_d}}{\pi} \right) (\phi_d - a_9) \|x_*\|^2 \left( 1 - \frac{\pi^2 a_9^2}{8 \phi_d^2} \right)\\
\leq& \frac{1}{2^{d+1}} \left( \phi_d^2 + 2 \phi_d a_9 + a_9^2 - 2 \left(1 - \frac{d a_9}{\phi_d}\right) (\phi_d -a_9) \left(1 - \frac{\pi^2 a_9^2}{8 \phi_d^2}\right) \right) \|x_*\|^2 \\
\leq& \frac{1}{2^{d+1}} \left( \phi_d^2 - 2 \phi_d + \frac{10}{a_8^3} d a_9 \right) \|x_*\|^2,
\end{align*}
where the last inequality is by Lemma~\ref{lemma:rho-d} and $a_9 < 1 / (4 \pi)$.

If $x \in \mathcal{B}(- \phi_d x_*, a_9 \|x_*\|)$, then we have
$0 \leq \pi - \bar\theta_{0, x, x_*} \leq \arcsin(a_9 \pi) \leq \frac{\pi^2}{2} a_9$, and
$\phi_d \|x_*\| - a_9 \|x_*\| \leq \|x\| \leq \phi_d \|x_*\| + a_9 \|x_*\|$.
It follows that
\begin{align*}
{f^E}(x) - \frac{\|x_*\|^2}{2^{d+1}} \geq& \frac{1}{2^{d+1}} \|x\|^2 - \frac{1}{2^d} \sum_{i = 0}^{d - 1} \frac{\sin \bar{\theta}_{i, x, x_*}}{\pi} \left( \prod_{j = i + 1}^{d-1} \frac{\pi - \bar{\theta}_{j, x, x_*}}{\pi} \right) \|x_*\|\|x\| \\
\geq& \frac{1}{2^{d+1}} \|x\|^2 - \frac{1}{2^d} \left( \rho_d + \frac{3 d^3a_9 \pi^2}{2} \right) \|x_*\|\|x\| \;\;\;\;\; \hbox{ (by~\cite[(41)]{HV17})} \\
\geq& \frac{1}{2^{d+1}} \left( \phi_d - a_9 \right)^2 \|x_*\|^2 - \frac{1}{2^d} \left( \rho_d + \frac{3 d^3a_9 \pi^2}{2} \right) \left(\phi_d + a_9\right)\|x_*\|^2 \\
\geq& \frac{1}{2^{d+1}} \left(\phi_d^2 - 2 \phi_d \rho_d - 10 d^3 a_9 \right) \|x_*\|^2.
\end{align*}
\end{proof}

Lemma~\ref{GA:le15} is used in proofs for Lemma~\ref{GA:le14}.

\begin{lemma} \label{GA:le15}
If
the WDC and RRIC
hold with $\epsilon < 1 / (16 \pi d^2)^2$, then we have
\begin{equation*}
|f(x) - {f^E}(x)| \leq \frac{\epsilon (1 + 4 \epsilon d)}{2^d} \|x\|^2 + \frac{\epsilon(1 + 4 \epsilon d) + 48 d^3 \sqrt{\epsilon}}{2^{d+1}} \|x\| \|x_*\| + \frac{\epsilon (1 + 4 \epsilon d)}{2^d} \|x_*\|^2.
\end{equation*}
\end{lemma}
\begin{proof}
For brevity of notation, let $\Lambda_{z} = \prod_{i = d}^1 W_{i, +, z}$.
We have
\begin{align*}
\left|f(x) - {f^E}(x) \right| =& \left|\frac{1}{2} x^T \left( \Lambda_{x}^T A^T A \Lambda_{x} - \Lambda_{x}^T \Lambda_{x} \right) x \right. + \frac{1}{2} x^T \left( \Lambda_{x}^T \Lambda_{x} - \frac{I_k}{2^d} \right) x \\
& - x^T \left( \Lambda_{x}^T A^T A \Lambda_{x_*} x_* -  \Lambda_{x}^T \Lambda_{x_*} x_* \right) - x^T \left( \Lambda_{x}^T \Lambda_{x_*} x_* - h_{x, x_*} \right) \\
&+ \frac{1}{2} x_*^T \left(\Lambda_{x_*}^T A^T A \Lambda_{x_*} - \Lambda_{x_*}^T \Lambda_{x_*} \right) x_* + \left. \frac{1}{2} x_*^T \left(\Lambda_{x_*}^T \Lambda_{x_*} - \frac{I_k}{2^d} \right) x_* \right| \\
&\leq \frac{\epsilon}{2} \frac{1 + 4 \epsilon d}{2^d} \|x\|^2 + \frac{\epsilon}{2} \frac{1 + 4 \epsilon d}{2^d} \|x\|^2 +  \frac{\epsilon}{2} \frac{1 + 4 \epsilon d}{2^d} \|x\| \|x_*\| + \frac{24 d^3 \sqrt{\epsilon}}{2^d} \|x\| \|x_*\| \\
&+ \frac{\epsilon}{2} \frac{1 + 4 \epsilon d}{2^d} \|x_*\|^2 + \frac{\epsilon}{2} \frac{1 + 4 \epsilon d}{2^d} \|x_*\|^2\\
=& \frac{\epsilon (1 + 4 \epsilon d)}{2^d} \|x\|^2 + \frac{\epsilon(1 + 4 \epsilon d) + 48 d^3 \sqrt{\epsilon}}{2^{d+1}} \|x\| \|x_*\| + \frac{\epsilon (1 + 4 \epsilon d)}{2^d} \|x_*\|^2,
\end{align*}
where the first inequality uses the WDC, the RRIC, and~\cite[Lemma~6]{HV17}.
\end{proof}

Lemma~\ref{GA:le20} is used in proofs for Lemma~\ref{GA:le19}.

\begin{lemma} \label{GA:le20}
Suppose $W \in \mathbb{R}^{n \times k}$ satisfies the WDC with constant $\epsilon$. Then for any $x, y \in \mathbb{R}^k$, it holds that
\begin{equation*}
\|W_{+, x} x - W_{+, y} y\| \leq \left( \sqrt{\frac{1}{2} + \epsilon} + \sqrt{2(2\epsilon + \theta)} \right) \|x - y\|,
\end{equation*}
where $\theta = \angle(x, y)$.
\end{lemma}
\begin{proof}
We have
\begin{align}
&\|W_{+, x} x - W_{+, y} y\| \leq \|W_{+, x} x - W_{+, x} y\| + \|W_{+, x} y - W_{+, y} y\| \nonumber \\
=& \|W_{+, x} (x - y)\| + \|(W_{+, x} - W_{+, y}) y\| \leq \|W_{+, x}\| \|x - y\| + \|(W_{+, x} - W_{+, y}) y\|. \label{GA:e33}
\end{align}
By WDC assumption, we have
\begin{align}
\|W_{+, x}^T (W_{+, x} - W_{+, y})\| \leq& \left\|W_{+, x}^T W_{+, x} - I  / 2\right\| + \left\|W_{+, x}^T W_{+, y} - Q_{x, y}\right\| + \left\| Q_{x, y} -  I / 2 \right\| \nonumber \\
\leq& 2 \epsilon + \theta. \label{GA:e34}
\end{align}
We also have
\begin{align}
&\|(W_{+, x} - W_{+, y}) y\|^2 = \sum_{i = 1}^n (1_{w_i \cdot x > 0} - 1_{w_i \cdot y > 0})^2 (w_i \cdot y)^2 \nonumber \\
\leq& \sum_{i = 1}^n (1_{w_i \cdot x > 0} - 1_{w_i \cdot y > 0})^2 ((w_i \cdot x)^2 + (w_i \cdot y)^2 - 2 (w_i \cdot x) (w_i \cdot y) ) \nonumber \\
=& \sum_{i = 1}^n (1_{w_i \cdot x > 0} - 1_{w_i \cdot y > 0})^2 (w_i \cdot (x - y) )^2 \nonumber \\
=& \sum_{i = 1}^n 1_{w_i \cdot x > 0} 1_{w_i \cdot y \leq 0} (w_i \cdot (x - y))^2 + \sum_{i = 1}^n 1_{w_i \cdot x \leq 0} 1_{w_i \cdot y > 0} (w_i \cdot(x - y))^2 \nonumber \\
=& (x - y)^T W_{+, x}^T (W_{+, x} - W_{+, y}) (x - y) + (x - y)^T W_{+, y}^T (W_{+, y} - W_{+, x}) (x - y) \nonumber \\
\leq& 2 (2 \epsilon + \theta) \|x - y\|^2. \;\;\; \hbox{ (by~\eqref{GA:e34})} \label{GA:e35}
\end{align}
Combining~\eqref{GA:e33},~\eqref{GA:e35}, and $\|W_{i, +, x}\|^2 \leq 1/2 + \epsilon$ given in~\cite[(10)]{HV17} yields the result.
\end{proof}

Lemma~\ref{GA:le19} is used in proofs for Lemma~\ref{GA:le10} and Lemma~\ref{GA:le8}.

\begin{lemma} \label{GA:le19}
Suppose $x \in \mathcal{B}(x_*, d \sqrt{\epsilon} \|x_*\|)$, and
the WDC
holds with $\epsilon < 1/ (200)^4 / d^6$. Then it holds that
\begin{equation*}
\left\|\prod_{i = j}^1 W_{i, +, x} x - \prod_{i = j}^1 W_{i, +, x_*} x_*\right\| \leq \frac{1.2}{2^{\frac{j}{2}}} \|x - x_*\|.
\end{equation*}
\end{lemma}
\begin{proof}
In this proof, we denote $\theta_{i, x, x_*}$ and $\bar{\theta}_{i, x, x_*}$ by $\theta_i$ and $\bar{\theta}_{i}$ respectively.
Since $x \in \mathcal{B}(x_*, d \sqrt{\epsilon} \|x_*\|)$, we have
\begin{equation} \label{GA:e40}
\bar{\theta}_{i} \leq \bar{\theta}_{0} \leq 2 d \sqrt{\epsilon}.
\end{equation}
By~\cite[(14)]{HV17}, we also have
$|\theta_{i} - \bar{\theta}_{i}| \leq 4 i \sqrt{\epsilon} \leq 4 d \sqrt{\epsilon}$.
It follows that
\begin{align}
2 \sqrt{\theta_i + 2 \epsilon} \leq& 2 \sqrt{\bar{\theta}_i + 4 d \sqrt{\epsilon} + 2\epsilon} \leq 2 \sqrt{2 d \sqrt{\epsilon} + 4 d \sqrt{\epsilon} + 2\epsilon} \nonumber \\
\leq& 2 \sqrt{8 d \sqrt{\epsilon}} \leq \frac{1}{30 d}. \hbox{ (by the assumption on $\epsilon$)} \label{GA:e39}
\end{align}
Note that $\sqrt{1 + 2 \epsilon} \leq 1 + \epsilon \leq 1 + \sqrt{d\sqrt{\epsilon}}$. We have
\begin{align*}
\prod_{i = d - 1}^0 \left( \sqrt{1+2\epsilon} + 2 \sqrt{ \theta_i + 2 \epsilon} \right) \leq& \left(1 + 7 \sqrt{d\sqrt{\epsilon}}\right)^d \leq 1 + 14 d \sqrt{d\sqrt{\epsilon}} \leq \frac{107}{100} < 1.2,
\end{align*}
where the second inequality is from that $(1+x)^d \leq 1 + 2dx$ if $0 < x d < 1$.
Combining the above inequality with Lemma~\ref{GA:le20} yields
\begin{align*}
\left\|\prod_{i = j}^1 W_{i, +, x}x - \prod_{i = j}^1 W_{i, +, x_*} x_*\right\| \leq \prod_{i = j - 1}^0 \left( \sqrt{\frac{1}{2}+\epsilon} + \sqrt{2} \sqrt{ \theta_i + 2 \epsilon} \right) \|x - x_*\| \leq \frac{1.2}{2^{\frac{j}{2}}} \|x - x_*\|.
\end{align*}
\end{proof}

Lemma~\ref{GA:le8} is used in proofs for Lemma~\ref{GA:le10}.

\begin{lemma}\label{GA:le8}
Suppose $x \in \mathcal{B}(x_*, d \sqrt{\epsilon} \|x_*\|)$, and
the WDC
holds with $\epsilon < 1/ (200)^4 / d^6$.  Then it holds that
\begin{equation*}
\left(\prod_{i = d}^1 W_{i, +, x}\right)^T\left[\left(\prod_{i = d}^1 W_{i, +, x}\right) x - \left(\prod_{i = d}^1 W_{i, +, x_*}\right) x_*\right] = \frac{1}{2^d} (x - x_*) + \frac{1}{2^d} \frac{1}{16} \|x - x_*\| O_1(1).
\end{equation*}
\end{lemma}
\begin{proof}
For brevity of notation, let $\Lambda_{j, k, z} = \prod_{i = j}^k W_{i, +, z}$.
We have
\begin{align}
&\Lambda_{d, 1, x}^T\left(\Lambda_{d, 1, x} x - \Lambda_{d, 1, x_*} x_*\right) \nonumber \\
=& \Lambda_{d, 1, x}^T\left[\Lambda_{d, 1, x} x - \sum_{j = 1}^d \left(\Lambda_{d, j, x} \Lambda_{j-1, 1, x_*} x_*\right)\right. + \left.\sum_{j = 1}^d \left(\Lambda_{d, j, x} \Lambda_{j-1, 1, x_*} x_*\right) - \Lambda_{d, 1, x_*} x_*\right] \nonumber \\
=& \underbrace{\Lambda_{d, 1, x}^T \Lambda_{d, 1, x} (x - x_*)}_{T_1} + \underbrace{\Lambda_{d, 1, x}^T \sum_{j = 1}^d \Lambda_{d, j + 1, x} \left( W_{j, +, x} - W_{j, +, x_*} \right) \Lambda_{j-1, 1, x_*} x_*}_{T_2}. \label{GA:e37}
\end{align}
For $T_1$, we have
\begin{align} \label{GA:e54}
T_1 = \frac{1}{2^d} (x - x_*) + \frac{4 d} {2^d} \|x - x_*\| O_1(\epsilon). \;\; \hbox{ (\cite[(10)]{HV17})}
\end{align}
For $T_2$, we have
\begin{align}
T_2 =& O_1(1) \sum_{j = 1}^d \left( \frac{1}{2^{d - \frac{j}{2}}} + \frac{(4d - 2j)\epsilon}{2^{d - \frac{j}{2}}} \right) \left\|(W_{j, +, x} - W_{j, +, x_*}) \Lambda_{j-1, 1, x_*}x_* \right\| \nonumber \\
=& O_1(1) \sum_{j = 1}^d \left( \frac{1}{2^{d - \frac{j}{2}}} + \frac{(4d - 2j)\epsilon}{2^{d - \frac{j}{2}}} \right) \left\|(\Lambda_{j-1, 1, x} x - \Lambda_{j-1, 1, x_*} x_*) \right\| \sqrt{2 (\theta_{i, x, x_*} + 2 \epsilon)} \nonumber \\
=& O_1(1) \sum_{j = 1}^d \left( \frac{1}{2^{d - \frac{j}{2}}} + \frac{(4d - 2j)\epsilon}{2^{d - \frac{j}{2}}} \right) \frac{1.2}{2^{\frac{j}{2}}} \|x - x_*\| \frac{1}{ 30 \sqrt{2} d} \nonumber \\
=& \frac{1}{16} \frac{1}{2^d} \|x - x_*\| O_1(1). \label{GA:e55}
\end{align}
where the first equation is by~\cite[(10)]{HV17}; the second equation is by~\eqref{GA:e35}; the third equation is by Lemma~\ref{GA:le19} and~\eqref{GA:e39}. The result follows from~\eqref{GA:e37}, \eqref{GA:e54} and~\eqref{GA:e55}.
\end{proof}

\end{document}